\input ./style/arxiv-general.cfg
\documentclass[aop,MSNbibl,seceqn,secthm,dvips]{arximspdf}
\makeatletter
   \@ifpackageloaded{graphicx}{}{\usepackage{graphicx}}
\makeatother

\doi{10.1214/14-AOP951} 
\volume{43}
\issue{5}
\pubyear{2015}
\firstpage{2763}
\lastpage{2809}
\docsubty{FLA}

\makeatletter
\newcommand{\ds}{\displaystyle}
\newcommand{\rrvert}{\vert}
\newcommand{\llvert}{\vert}
\newcommand{\eqref}[1]{(\ref{#1})}
\newtheorem{theorem}{Theorem}[section]
\newtheorem{corollary}[theorem]{Corollary}
\newtheorem{lemma}[theorem]{Lemma}
\newtheorem{proposition}[theorem]{Proposition}
\newproclaim{remark}[theorem]{Remark}
\newcommand{\pr}{\mathbf{P}}
\newcommand{\e}{\mathbf{E}}
\newcommand{\Z}{\mathbb{Z}}
\newcommand{\R}{\mathsf{R}}
\newcommand{\wZ}{\tilde{Z}}
\newcommand{\bDelta}{\bolds{\Delta}}
\newcommand{\bL}{\mathbf{L}}
\newcommand{\cN}{\mathcal{N}}
\newcommand{\1}{\mathsf{1}}
\makeatother

\begin{document}
\begin{frontmatter}

\title{Multifractal analysis of superprocesses\\ with~stable branching in dimension one}
\runtitle{Multifractal analysis of superprocesses}

\begin{aug}
\author[A]{\fnms{Leonid}~\snm{Mytnik}\corref{}\thanksref{T1}\ead[label=e1]{leonid@ie.technion.ac.il}}
\and
\author[B]{\fnms{Vitali}~\snm{Wachtel}\thanksref{T2}\ead[label=e2]{wachtel@mathematik.uni-muenchen.de}\ead[label=u1,url]{http://www.foo.com}}
\runauthor{L. Mytnik and V. Wachtel}
\affiliation{Technion---Israel Institute of Technology and University of Munich}
\address[A]{Faculty of Industrial Engineering\\
\quad and Management\\
Technion Israel Institute of
Technology\\
Haifa 32000\\
Israel\\
\printead{e1}}
\address[B]{Mathematical Institute\\
University of Munich\\
Theresienstrasse 39, D--80333\\
Munich\\
  Germany\\
\printead{e2}}
\end{aug}
\thankstext{T1}{Supported by the Israel Science Foundation grant.}
\thankstext{T2}{Supported by the German Israeli Foundation for Scientific Research and
Development, Grant  G-2241-2114.6/2009.}

\received{\smonth{1} \syear{2013}}
\revised{\smonth{7} \syear{2014}}

\begin{abstract}
We show that density functions of a $(\alpha,1,\beta)$-superprocesses are
almost sure multifractal for $\alpha>\beta+1$, $\beta\in(0,1)$ and calculate
the corresponding spectrum of singularities.
\end{abstract}

\end{frontmatter}

\section{Introduction, main results and discussion}\label{sec1}

For $0<\alpha\leq2$ and $1+\beta\in(1,2]$, the
so-called $(\alpha,d,\beta)$-\emph{superprocess} $X=\{X_{t}\dvtx  t\geq
0\}$ in $\mathsf{R}^{d}$ is a finite measure-valued process related
to the log-Laplace equation
%
\begin{equation}
\label{logLap}
\frac{d}{dt}u = \Delta _{\alpha}u +au-
bu^{1+\beta},
\end{equation}
where $ a\in\mathsf{R}$ and $ b>0$ are any fixed
constants. Its underlying motion is described by the fractional Laplacian
$\Delta_{\alpha}:=-(-\Delta)^{\alpha/2}$ determining a symmetric $\alpha$-stable motion in
$\mathsf{R}^{d}$ of index $ \alpha\in(0,2]$ (Brownian motion corresponds to
$ \alpha=2$), whereas its continuous-state branching mechanism
%
\begin{equation}
\label{not.Psi}
v \mapsto -av+bv^{1+\beta},\qquad v\geq0,
\end{equation}
belongs to the domain of attraction of a stable law of index $ 1+\beta
\in(1,2]$ (the branching mechanism is critical if $a=0$).

Let\vspace*{1pt} $d<\frac{\alpha}{\beta}$. Then, for any fixed time $t$, $X_t(dx)$ is  a.s. absolutely continuous with respect to the
Lebesgue measure (cf. Fleischmann \cite{Fleischmann1988.critical}
for $a=0$).

In the case of $\beta=1$, there is a continuous version of the density of $X_t(dx)$ for all $\alpha>1$;
see Konno and Shiga \cite{KonnoShiga1989}. A careful examination of their arguments shows that this
density is H\"older continuous with any exponent smaller than $(\alpha-1)/2$.

Now we consider the case $\beta<1$.
As shown in Fleischmann, Mytnik and Wachtel
(\cite{FMW10}, Theorem~1.2(a), (c)), there is
a \emph{dichotomy} for the  density function of the measure (in what follows, we just say the ``density of
 the measure''):
there is a~continuous version of the density of $X_{t}(dx)$  if
$d=1$ and $ \alpha>1+\beta$, but otherwise the density
is unbounded  on open sets of
positive $X_{t}(dx)$-measure. Note that the case of $\alpha=2$ had been studied
earlier in Mytnik and Perkins \cite{MP03}.
In the case of
continuity, H\"{o}lder regularity properties of the density had been
studied in \cite{FMW10,FMW11}.

From now on, we always assume that $d=1$, $\beta<1$ and $\alpha>1+\beta$, that is, there is a continuous
 version of the density at fixed time $t$. This density, with a slight abuse of notation, will be also denoted by
 $X_t(x), x\in \R$.

Let us first recall the notion of pointwise regularity (see,
e.g., Jaffard \cite{Jaff99}). We say that a function $f$
has regularity of index $\eta>0$ at a point $x\in \R$,
if there exists an open neighborhood $U(x)$ of $ x$, a constant $C>0$ and
a polynomial $P_x$ of degree at most $\lfloor \eta \rfloor$
such that
%
\begin{equation}
\label{eq9101}
\bigl|f(y)-P_x(y) \bigr| \leq C |y-x|^{\eta}\qquad
\mbox{for all }y\in U(x).
\end{equation}
For $\eta\in (0,1)$, the above definition coincides with the definition of H\"older continuity with index $\eta$ at
a point. Note that sometimes the class of functions satisfying~(\ref{eq9101}) is denoted by $C^{\eta}(x)$.
Now, given $f$ one would like to find the supremum over all $\eta$ such that~(\ref{eq9101}) holds for
 some constant $C$ and polynomial $P_x$. This leads to the definition of so-called
\textit{optimal} H\"older  exponent (or index) of $f$ at $x$:
%
\begin{equation}
H_f(x):= \sup \bigl\{\eta>0\dvtx f\in C^{\eta}(x) \bigr\},
\end{equation}
and we set it to $0$ if $f \notin C^{\eta}(x)$ for all $\eta>0$. To simplify the exposition, we will
sometimes call $H_f(x)$ the H\"older exponent of $f$ at $x$.

Let us fix $t>0$ and  return to the continuous density $X_{t}$ of the
$(\alpha,1,\beta)$-superprocess.
In what follows, $H_X(x)$ will denote the optimal
H\"{o}lder exponent of $X_{t}$ at $x\in \R$. In Theorem~1.2(a), (b)
of~\cite{FMW10}, the
so-called \emph{optimal index} for \emph{local} H\"{o}lder continuity of
$X_{t}$ had been determined as
%
\begin{equation}
\label{etac}
\eta_{\mathrm{c}} := \frac{\alpha}{1+\beta}-1 \in (0,1).
\end{equation}
This means that $\inf_{x\in K} H_X (x)\geq\eta_{\mathrm{c}}$ for any compact $K$ and, moreover,
in any nonempty open set $U\subset\mathsf{R}$ with
$X_{t}(U)>0$ one can find (random) points $x$ such that $H_X(x)=\eta
_{\mathrm{c}}$.
Moreover, it was proved in~\cite{FMW11} that for any fixed point $x\in \R$ we have
\[
H_X(x)= \bar{\eta}_{\mathrm{c}} := \frac{1+\alpha}{1+\beta}-1\qquad
\mbox{a.s. on } \bigl\{X_t(x)>0 \bigr\}
\]
in the case of $\beta>(\alpha-1)/2$, and
\[
H_X(x)\geq 1 \qquad \mbox{a.s. on } \bigl\{X_t(x)>0
\bigr\}
\]
if $\beta\leq(\alpha-1)/2$.

\begin{remark}
In~\cite{FMW11}, the classical definition of H\"older exponent was used, which can take only values between
 $0$ and $1$. Hence, the result in~\cite{FMW11}  states that the optimal index of H\"older continuity (in classical
 sense) equals  $\min \{\frac{1+\alpha}{1+\beta
}-1, 1\}$, for any $\beta\in (0,\alpha-1)$.
\end{remark}

The purpose of this paper  includes proving that on any open set of positive $X_t$ measure, and for any $\eta\in
(\eta_{\mathrm{c}},\bar{\eta}_{\mathrm{c}})\setminus\{1\}$ there are, with probability one,
(random) points $x\in\R$ such that the optimal H\"{o}lder index $H_X(x)$
of $X_{t}$ at $x$ is exactly $\eta$.
Moreover, for an open set  $U\subset \R$,  we are going to
determine the \emph{Hausdorff dimension}, say $D_{U}(\eta)$, of the (random) set
\[
\mathcal{E}_{U,X,\eta}:= \bigl\{x\in U\dvtx H_X(x)=\eta \bigr
\}.
\]
We will show that the function $ \eta\mapsto
D_{U}(\eta)$ is independent of $U$;  it reveals the so-called
\emph{multifractal spectrum} related to the optimal H\"{o}lder index at points.

To formulate our main result, we need also the following notation. Let
$\mathcal{M}_{\mathrm{f}}$ denote the set of finite measures on $\mathsf{R}$, and
for $\mu\in \mathcal{M}_{\mathrm{f}}$, $|\mu|$ will denote the total mass $\mu(\R)$.
Our main result is as follows.

\begin{theorem}[(Multifractal spectrum)]\label{thmmfractal}
Fix $t>0$,  and $X_{0}=\mu \in\mathcal{M}_{\mathrm{f}}$.
Let $d=1$ and
$\alpha>1+\beta$.
Then,
for any $\eta\in [
\eta_{\mathrm{c}}, \bar\eta_{\mathrm{c}})\setminus\{1\}$, with probability one,
%
\[
D_U(\eta)= (\beta+1) (\eta-\eta_{\mathrm{c}}) \qquad
\mbox{for any open set }U\subset \R,
\]
whenever  $X_t(U)>0$.
\end{theorem}

\begin{remark}
Let us consider the case $\eta=\bar\eta_{\mathrm{c}}$. First note that
if $\bar\eta_{\mathrm{c}}<1$ then, for every fixed $x$, $H_X(x)=\bar\eta_{\mathrm{c}}$
almost surely on the event $\{X_t(x)>0\}$, see Theorem~1.1 from \cite{FMW11}. We conjecture
that this statement is also valid whenever $\bar\eta_{\mathrm{c}}\geq 1$. This
would imply that $D_U(\bar\eta_{\mathrm{c}})=1$ almost surely on $\{X_t(U)>0\}$.
Indeed, for $B$~being an arbitrary interval in $(0,1)$ define
\[
\lambda({\mathcal E}_{B,X,\bar\eta_{\mathrm{c}}}) = \int_B
1_{\{H_X(x)= \bar\eta_{\mathrm{c}}\}}\, dx,
\]
where $\lambda$ is the Lebesgue measure on $\R$. Then, by the Fubini theorem, we have
\begin{eqnarray*}
\e \Bigl[\lambda({\mathcal E}_{B,X,\bar\eta_{\mathrm{c}}})\big|\inf_{y\in B}X_t(y)>0
\Bigr]&=& \e \biggl[ \int_B 1_{\{H_X(x)= \bar\eta_{\mathrm{c}}\}} \,dx \Big|\inf
_{y\in B}X_t(y)>0 \biggr]
\\
&=& \int_B \pr \Bigl(H_X(x)= \bar
\eta_{\mathrm{c}} \big|\inf_{y\in B}X_t(y)>0 \Bigr)\,
dx
\\
&=&\lambda(B),
\end{eqnarray*}
in the last step we used our conjecture that
$H_X(x)= \bar\eta_{\mathrm{c}}$ with probability one for every fixed point $x$. That is, given
$\{\inf_{y\in B}X_t(y)>0\}$, we get that, with probability one,
$D_{B}(\bar\eta_{\mathrm{c}})=1$. We may fix $\omega$ outside a $\pr$-null set
so that this holds for any rational ball $B$, that is, for any ball with a
rational radius and center. Let $U$ be an arbitrary open set such that
$\{\inf_{y\in U}X_t(y)>0\}$. Then there is always a rational ball $B$ in
$U$ such that  $\{\inf_{y\in B}X_t(y)>0\}$, and so the result follows
immediately from what we derived for the fixed ball.
\end{remark}

\begin{remark}
The fact that our proof fails in the case $\eta=1$ is even more disappointing.
Formally, it happens for some technical reasons, but one has also to note, that
this point is critical: it is the borderline between differentiable and nondifferentiable functions.
However, we still believe that the function $D_U(\cdot)$ can be continuously extended to $\eta=1$, that is,
$D_U(1)=(\beta+1)(1-\eta_{\mathrm{c}})$ almost surely on $\{X_t(U)>0\}$.
\end{remark}

\begin{remark}
The condition $\alpha>1+\beta$ excludes the case of the quadratic
super-Brownian motion, that is, $\alpha=2$, $\beta=1$. But it is a
known ``folklore'' result that the super-Brownian motion $X_t(\cdot)$ is almost
surely monofractal on any open set of strictly positive density.
That is, $\pr$-a.s., for any $x$ with $X_t(x)>0$ we have $H_X(x)=1/2$.
For the fact that $H_X(x)\geq1/2$, for any $x$, see Konno and Shiga \cite{KonnoShiga1989}
and Walsh~\cite{Walsh1986}. To get that $H_X(x)\leq 1/2$ on the event
$\{X_t(x)>0\}$, one can show that
\[
\limsup_{\delta\to0}\frac{|X_t(x+\delta)-X_t(x)|}{\delta^\eta}=\infty \qquad \mbox{for all }x
\mbox{ such that }X_t(x)>0, \pr\mbox{-a.s.},
\]
for every $\eta>1/2$. This result follows from the fact that for $\beta=1$
the noise driving the corresponding stochastic equation for $X_t$ is Gaussian
(see (0.4) in \cite{KonnoShiga1989}) in contrast to the case of $\beta<1$
considered here, where we have driving discontinuous noise with L\'evy type
intensity of jumps.
\end{remark}

The multifractal spectrum of random functions and measures has attracted
attention for many years and has been studied, for example, in Dembo et
al. \cite{DemboPeresRosenZeitouni2001}, Durand \cite{Durand2009}, Hu and
Taylor \cite{HuTaylor2000}, Klenke and M\"{o}rters \cite{KlenkeMoerters2005},
Le Gall and Perkins \cite{LeGallPerkins1995}, M\"{o}rters and Shieh
\cite{MoertersShieh2004} and Perkins and Taylor \cite{PerkinsTaylor1998}. The
multifractal spectrum of singularities that describes the Hausdorff dimension
of sets of different H\"{o}lder exponents of functions was investigated for
deterministic and random functions in Jaffard
\cite{Jaff99,Jaffard2000,Jaffard2004} and Jaffard and Meyer
\cite{JaffardMeyer1996}.

We now turn to the description of our approach.
We would like to verify the spectrum of singularities of $X_t(\cdot)$
on any open (random) set $U$ whenever $X_t(U)>0$. Based on the ideas
of the proof of Theorem~1.1(b) in~\cite{MP03}, it is enough to verify
the spectrum of singularities of $X_t(\cdot)$ on any \textit{fixed} open
ball $U$ in $\R$.  In what follows, we fix, without loss of generality,
$U=(0,1)$. The extension of our argument to general open $U$ is trivial.

Next, we derive a representation for the density $X_t(\cdot)$, which will
be used in the proof.
Let $ p^{\alpha}$ denote the continuous $\alpha$-stable
transition kernel related to the fractional Laplacian
$\Delta_{\alpha}=-(-\Delta)^{\alpha/2}$ in $\mathsf{R}$, and $(S^{\alpha}_t,t\geq0)$ the related
semigroup, that is,
\[
S^{\alpha}_tf(x)=\int_{\R}p^\alpha(x-y)f(y)
\,dy\qquad \mbox{for any bounded function }f
\]
and
\[
S^{\alpha}_t\nu(x)=\int_{\R}p^\alpha(x-y)
\nu(dy)\qquad \mbox{for any finite measure }\nu.
\]
Fix $X_{0}=\mu\in\mathcal{M}_{\mathrm{f} }\setminus\{0\}$.
First, we want to recall the \emph{martingale decomposition} of the
$(\alpha,1,\beta)$-superprocess $X$ (valid for any $\alpha\in (0,2],\beta\in (0,1)$; see,
e.g., \cite{FMW10}, Lemma~1.6): for all
sufficiently smooth bounded nonnegative functions $\varphi$ on~$\mathsf{R}$ and $t\geq0$,
%
\begin{equation}
\label{mart.dec}
\langle X_{t},\varphi \rangle = \langle \mu,\varphi
\rangle +\int_{0}^{t} ds \langle
X_{s}, \Delta _{\alpha}\varphi \rangle +M_{t}(
\varphi)+a I_{t}(\varphi)
\end{equation}
with discontinuous martingale
%
\begin{equation}
\label{mart}
t \mapsto M_{t}(\varphi) := \int_{(0,t]\times\mathsf{R}\times
\mathsf{R}_{+}}
\tilde{\cN} \bigl( d(s,x,r) \bigr) r \varphi(x)
\end{equation}
and increasing process
%
\begin{equation}
\label{incr} t \mapsto I_{t}(\varphi) := \int_{0}^{t}
ds \langle X_{s},\varphi \rangle.
\end{equation}
Here, $\tilde{\cN}:=\cN-\widehat{\cN}$, where $\cN  (d
(s,x,r) ) $ is a random measure on $(0,\infty)\times\mathsf{R}
\times(0,\infty)$ describing all the jumps $ r\delta_{x}$ of $X$ at times $ s$ at sites $x$
 of size $ r$ (which are the only discontinuities of
the process $X$). Moreover,
%
\begin{equation}
\label{decomp}
\widehat{\cN} \bigl( d(s,x,r) \bigr) = \varrho
\,ds X_{s}(dx) r^{-2-\beta}\,dr
\end{equation}
is the compensator of $ \cN$, where $ \varrho:=b(1+\beta
)\beta/\Gamma(1-\beta)$ with $\Gamma$ denoting the Gamma function.

Under our assumptions, the random measure $ X_{t}(dx)$ is a.s. absolutely continuous for every fixed $t>0$.
\textrm{F}rom the Green function representation related to (\ref{mart.dec})
(see, e.g., \cite{FMW10}, (1.9)), we obtain the following
representation of a version of the density function of $ X_{t}(dx)$
(see, e.g., \cite{FMW10}, (1.12)):
%
\begin{eqnarray}
X_{t}(x) &=& \mu \ast p_{t}^{\alpha} (x) + \int
_{(0,t]\times
\mathsf{R}} M \bigl(d(s,y) \bigr) p_{t-s}^{\alpha}(y-x)
\nonumber\\
&&\label{rep.dens} {}+ a \int_{(0,t]\times\mathsf{R}} I \bigl(d(s,y) \bigr)
p_{t-s}^{\alpha}(y-x)\\
&=:& Z^{1}(x)+Z^{2}(x)+Z^{3}(x),
\qquad x\in\mathsf{R}\nonumber
\end{eqnarray}
(with notation in the obvious correspondence).
Note that although $Z^i, i=1,2,3$, depend on $t$,  we omit the corresponding subscript since $t$
is fixed throughout the paper.
$M (d(s,y) )$
in~(\ref{rep.dens}) is the martingale measure related to (\ref{mart})
and $I  (d(s,y) )$ the random measure
related to (\ref{incr}). Note that by Lemma~1.7 of~\cite{FMW10} the class of  ``legitimate''
integrands
with respect to the martingale measure $M  (d(s,y) )$
includes the set of functions $\psi$ such that for some $p\in (1+\beta, 2)$,
%
\begin{equation}
\label{eq11101}
\int_0^T ds \int_{\mathsf{R}} dx S^{\alpha}_s\mu(x)\bigl|
\psi(s,x)\bigr|^p <\infty\qquad \forall T>0.
\end{equation}
We let ${\mathcal L}^p_{\mathrm{loc}}$ denote the space of  equivalence classes of
measurable functions satisfying~(\ref{eq11101}). For $\alpha>1+\beta$,
it is easy to check that, for any $t>0, z\in \R$,
 $(s,x)\mapsto p^{\alpha}_{t-s}(z-x)\1_{\{s<t\}}$ is in  ${\mathcal L}^p_{\mathrm{loc}}$ for any
$p\in (1+\beta, 2)$, and hence the stochastic integral in the representation~(\ref{rep.dens}) is
well defined.

$Z^1$ is obviously twice differentiable. Moreover, it turns out (see
Corollary~\ref{Ztricorr}) that $Z^3$ is H\"older continuous of index
$\alpha{\mathsf{1}}_{\{\bar{\eta}_{\mathrm{c}}>1\}}+{\mathsf{1}}_{\{\bar{\eta}_{\mathrm{c}}\leq1\}}$.
Noting that this index is not smaller that $\overline{\eta}_{\mathrm{c}}$, we conclude
that the multifractal structure of $X_t$ coincides with that of $Z^2$. Recalling
the definitions of $Z^2$ and $M(ds,dy)$, we see that there is a~``competition''
between branching and motion: jumps of the martingale measure~$M$ try to destroy
smoothness of $X_t(\cdot)$ and $p^\alpha$ tries to make $X_t(\cdot)$ smoother. Thus, it is natural
to expect that $\{x\dvtx H_{Z^2}(x)=\eta\}$ can be described by jumps of a~certain
order depending on~$\eta$.

In Section~\ref{sec4}, we construct a random set $S_\eta$ such that
$\operatorname{dim}(S_\eta)\leq(1+\beta)(\eta-\eta_c)$ and
$\{H_{Z^2}(x)=\eta\}\subseteq S_{\eta+\varepsilon}$.
This,
after some simple manipulations, allows us to obtain the bound
$\operatorname{dim}(\{H_{Z^2}(x)=\eta\})\leq (1+\beta)(\eta-\eta_{\mathrm{c}})$.

It turns out that $S_\eta$ is not very convenient for the derivation
of the lower bound for $\operatorname{dim}(\{H_{Z^2}(x)=\eta\})$. For this
reason, in Section~\ref{sec5}, we introduce an alternative random set
$\tilde{J}_{\eta,1}$ with $\operatorname{dim}(\tilde{J}_{\eta,1})\geq  (\beta+1)(\eta-\eta_{\mathrm{c}})$, on
which we show existence of jumps which should lead to $H_{Z^2}(x)\leq\eta$
for $x\in\tilde{J}_{\eta,1}$. However, if several jumps occur in a
proximity of a point $x$ then they can compensate each other. To show that
this possible scenario  does not affect the Hausdorff dimension of the set,
$\{H_{Z^2}(x)=\eta\}$ is the most difficult part of our proof. This is
not unexpected: in our previous papers \cite{FMW10,FMW11}, the proofs of
the optimality of H\"older indices were much harder than the derivation of the
H\"older continuity. More precisely, we prove that such a compensation is
possible only  on a set of the Hausdorff dimension strictly smaller than
$(\beta+1)(\eta-\eta_{\mathrm{c}})$, and hence this does not influence the dimension result.

\section{Preliminaries}
\label{sec3}
\subsection{Estimates for the transition kernel of the \texorpdfstring{$\alpha$}{alpha}-stable motion}

The symbol~$C$ will always
denote a generic positive constant, which might change\vadjust{\goodbreak} from line  to line. On
the other hand, $C_{(\#)}$ denotes a constant appearing in formula line (or array)~$(\#)$.

Throughout the paper, we will need the following bound; see \cite{FMW10}, Lemma~2.1.

\begin{lemma}
\label{L1old}
For every $\delta\in[0,1]$,  there exists a constant $C>0$ such that
%
\begin{equation}
\label{L1.1}
\qquad \hspace*{3pt}\bigl|p_{t}^{\alpha}(x)-p_{t}^{\alpha}(y)
\bigr| \leq C \frac{|x-y|^{\delta
}}{t^{\delta/\alpha}} \bigl(p_{t}^{\alpha}(x/2)+p_{t}^{\alpha}
(y/2) \bigr), \qquad t>0, x,y\in\mathsf{R}.
\end{equation}
\end{lemma}

By methods very similar to those used for the proof of the previous lemma, one can get the
following result.

\begin{lemma}
\label{lemdensnew}
\textup{(a)}  For every $\delta
\in[0,1]$,  there exists a constant $C>0$ such that, for all $t>0$ and $x,y\in\mathsf{R}$,
%
\begin{equation}
\label{equt10105} \biggl\llvert \frac{\partial p_{t}^{\alpha}(x)}{\partial x}-\frac{\partial p_{t}^{\alpha}(y)}{\partial y} \biggr\rrvert
\leq C \frac{|x-y|^{\delta
}}{t^{(1+\delta)/\alpha}} \bigl(p_{t}^{\alpha}(x/2)+p_{t}^{\alpha}
(y/2) \bigr).
\end{equation}

\textup{(b)} There exists a constant $C>0$ such that
%
\begin{equation}
\label{equt10107} \biggl\llvert \frac{\partial p_{t}^{\alpha}(x)}{\partial x} \biggr\rrvert \leq C
t^{-1/\alpha} p_{t}^{\alpha}(x/2), \qquad t>0, x\in
\mathsf{R}.
\end{equation}
\end{lemma}

An immediate corollary from the above lemma is as follows.

\begin{corollary}
\label{cor14101}
For every $\delta
\in[ 1,2]$,
%
\begin{eqnarray}
&& \biggl\llvert p_{t}^{\alpha}(x)-p_{t}^{\alpha}(y)
-(x-y)\frac{\partial p_{t}^{\alpha}(y)}{\partial y} \biggr\rrvert
\nonumber\\[-8pt]
\label{equt10106}
\\[-8pt]
\nonumber
&&\qquad \leq C \frac{|x-y|^{\delta
}}{t^{\delta/\alpha}} \bigl(p_{t}^{\alpha}(x/2)+p_{t}^{\alpha}%
(y/2)
\bigr),\qquad t>0,x,y\in\mathsf{R}.
\end{eqnarray}
\end{corollary}

With Lemma~\ref{lemdensnew}(b) at hand, it is easy to check
that if $\beta<(\alpha-1)/2$, then for any $t>0, z\in \R$,
$(s,x)\mapsto \frac{\partial p^{\alpha}_{t-s}(z-x)}{\partial x}\1_{\{s<t\}}$
is in  ${\mathcal L}^p_{\mathrm{loc}}$ for any $p\in (1+\beta, \frac{1+\alpha}{2})$.
Then, using again condition~(\ref{eq11101}), it is easy to show the
following result.

\begin{lemma}
\label{lem11101}
Let $\beta<(\alpha-1)/2$. Then for any fixed $t>0, x\in\R$, the stochastic integral
\[
\int_{(0,t]\times
\mathsf{R}} M \bigl(d(s,y) \bigr) \frac{\partial p_{t-s}^{\alpha}(y-x)}{\partial x}
\]
is well defined.
\end{lemma}

In what follows, we let $\frac{\partial Z^{2}(x)}{\partial x}$ denote the integral
$\int_{(0,t]\times
\mathsf{R}} M  (d(s,y) )
\frac{\partial p_{t-s}^{\alpha}(y-x)}{\partial x}$.

\subsection{Bound for stable processes}

Let
$L=\{L_{t}\dvtx  t\geq0\}$ denote a spectrally positive stable process of index
$\kappa\in(1,2)$.  That is, $L$ is an $\mathsf{R}$-valued
time-homogeneous process with independent increments and with Laplace
transform given by
%
\begin{equation}
\label{Laplace} \mathbf{E} e^{-\lambda L_{t}} = e^{t\lambda^{\kappa}},
\qquad \lambda,t\geq0.
\end{equation}
Let
$\bDelta L_{s}:=L_{s}-L_{s-}>0$ denote the (positive) jumps of
$ L$. The next technical result gives an exponential upper bound for the tail
of $\sup_{0\leq u\leq t}|L_{u}|$ under the condition that all the jumps of $L$ are not
too large.

\begin{lemma}
\label{L3}
There exists a constant $C_{(\fontsize{8.36}{10.36}{\selectfont\ref{eq07033}})}=C_{(\fontsize{8.36}{10.36}{\selectfont\ref{eq07033}})}(\kappa)$ such that
%
\begin{eqnarray}
&& \mathbf{P} \Bigl( \sup_{0\leq u\leq t}|L_{u}|
\mathsf{1}_{ \{\sup_{0\leq v\leq
u}\bDelta L_{v}\leq y \}}\geq x \Bigr)
\nonumber
\\[-8pt]
\label{eq07033}
\\[-8pt]
\nonumber
&& \qquad \leq \biggl(\frac{C_{(\fontsize{8.36}{10.36}{\selectfont\ref{eq07033}})} t}{xy^{\kappa
-1}} \biggr)^{ x/y}+\exp
\biggl\{-\frac{x^{\kappa/(\kappa-1)}}{C_{(\fontsize{8.36}{10.36}{\selectfont\ref{eq07033}})}t^{1/(\kappa-1)}} \biggr\}
\end{eqnarray}
and
%
\begin{equation}
\label{eq07033a} \mathbf{P} \Bigl( \sup_{0\leq u\leq t}L_{u}
\mathsf{1}_{ \{\sup_{0\leq v\leq
u}\bDelta L_{v}\leq y \}}\geq x \Bigr) \leq \biggl(\frac{C_{(\fontsize{8.36}{10.36}{\selectfont\ref{eq07033}})} t}{xy^{\kappa
-1}}
\biggr)^{ x/y}
\end{equation}
for all $t,x,y>0$.
\end{lemma}

This bound \eqref{eq07033a} has been proven in \cite{FMW10} (see Lemma~2.3 there).
To prove the inequality for $\sup_{0\leq u\leq t}|L_{u}|$, one has to combine \eqref{eq07033a}
with Lemma~2.4 from \cite{FMW10}.

\subsection{Analysis of $Z^1$ and $Z^3$}

Consider $Z^1, Z^3$ on the right-hand  side of~(\ref{rep.dens}). Clearly, $Z^1$ is twice differentiable. Noting that $\bar{\eta}_{\mathrm{c}}<2$
for all $\alpha,\beta$, we see that $Z^1$ does not affect the optimal H\"older exponent of $X_t$. As for $Z^3$,
we have the following result.

\begin{lemma}\label{Ztri}
Let $\beta<(\alpha-1)/2$. Then $\mathbf{P}$-a.s.,
 $Z^3(x)$ is differentiable for any $x\in (0,1)$, and the mapping
\[
x\mapsto \frac{d}{dx}Z^3(x), \qquad x\in (0,1),
\]
is, $\mathbf{P}$-a.s., H\"older continuous with any exponent
$\eta<\alpha-1$.
\end{lemma}

\begin{pf}
Using Lemma~2.12 in \cite{FMW10} with $\theta=\delta=1$, we see that $Z^3(x)$ is
differentiable and, furthermore,
\[
\frac{d}{dx}Z^3(x)=a\int_0^t
ds \int_{\mathsf{R}}X_s(dy)\frac{\partial}{\partial
x}p_{t-s}^\alpha(x-y).
\]
Therefore, for any $x_1,x_2\in (0,1)$,
\begin{eqnarray*}
&& \biggl|\frac{d}{dx}Z^3(x_1)-\frac{d}{dx}Z^3(x_2)
\biggr|
\\
&& \qquad \leq |a|\int_0^t ds \int
_{\mathsf{R}}X_s(dy) \biggl|\frac{\partial}{\partial
x_1}p_{t-s}^\alpha(x_1-y)-
\frac{\partial}{\partial x_2}p_{t-s}^\alpha(x_2-y) \biggr|.
\end{eqnarray*}
Applying now Lemma~\ref{lemdensnew} with $\delta<\alpha-1$, we obtain
\begin{eqnarray*}
&& \biggl|\frac{d}{dx}Z^3(x_1)-\frac{d}{dx}Z^3(x_2)
\biggr|
\\
&& \qquad \leq C|a||x_1-x_2|^\delta \int
_0^t ds (t-s)^{-(1+\delta)/\alpha}
\\
&&\qquad \quad{}\times \int_{\mathsf{R}}X_s(dy)
\biggl(p_{t-s}^\alpha \biggl(\frac{x_1-y}{2}
\biggr)+p_{t-s}^\alpha \biggl(\frac{x_2-y}{2} \biggr)
\biggr)
\\
&& \qquad = C|a||x_1-x_2|^\delta\int
_0^t ds (t-s)^{-(1+\delta)/\alpha}
\\
&& \qquad \quad {}\times \bigl(S^\alpha_{2^\alpha(t-s)}X_s(x_1)+S^\alpha_{2^\alpha(t-s)}X_s(x_2)
\bigr)
\\
&& \qquad \leq C|a||x_1-x_2|^\delta\sup
_{s\leq t, x\in
(0,1)}S^\alpha_{2^\alpha(t-s)}X_s(x)\int
_0^t ds s^{-(1+\delta)/\alpha}
\\
&& \qquad = C\alpha(\alpha-1-\delta)^{-1}|a||x_1-x_2|^\delta
\sup_{s\leq t, x\in
(0,1)}S^\alpha_{2^\alpha(t-s)}X_s(x).
\end{eqnarray*}
Taking into account Lemma~2.11 from \cite{FMW10}, which states that
%
\begin{equation}
\label{eq08021}
V:=\sup_{s\leq t, x\in (0,1)}S^\alpha_{2^\alpha(t-s)}X_s(x)<
\infty, \qquad \mathbf{P}\mbox{-a.s.},
\end{equation}
we see that $x\mapsto \frac{d}{dx}Z^3(x)$ is H\"older continuous with the exponent $\delta$.
\end{pf}

Combining this lemma with \cite{FMW10}, Remark~2.13, we obtain

\begin{corollary}\label{Ztricorr}
$\mathbb{P}$-a.s., for any $x\in (0,1)$ we have
%
\begin{equation}
\label{Ztricorr1} H_{Z^3}(x)\geq\alpha{\mathsf{1}}_{\{\bar{\eta}_c>1\}}+{
\mathsf{1}}_{\{\bar{\eta}_c\leq1\}}.
\end{equation}
\end{corollary}

From this corollary and the fact that the right-hand side in (\ref{Ztricorr1})
is not smaller than $\bar{\eta}_c$, we conclude that $Z^3$ does not affect
the multifractal structure of $X_t$ either. More precisely, the
spectrum of singularities of $X_t$ coincides
with that of $Z^2$. Consequently, to prove Theorem~\ref{thmmfractal}, we have to determine Hausdorff dimensions
of the sets
\begin{eqnarray*}
\mathcal{E}_{Z^2,\eta}&:= & \bigl\{x\in (0,1)\dvtx H_{Z^2}(x)=
\eta \bigr\},
\\
\tilde{\mathcal E}_{Z^2,\eta}&:=& \bigl\{x\in (0,1)\dvtx
H_{Z^2}(x)\leq\eta \bigr\},
\end{eqnarray*}
and this is done in the next two sections.

\section{Upper bound for the Hausdorff dimension}
\label{sec4}

The aim of this section is to prove the following proposition.

\begin{proposition}
\label{prop0802}
For every $\eta\in[\eta_c,\overline{\eta}_c)$,
\[
\operatorname{dim}({\mathcal E}_{Z^2,\eta})\leq \operatorname{dim}(\tilde{
\mathcal E}_{Z^2,\eta})\leq (1+\beta) (\eta-\eta_c), \qquad
\mathbf{P}\mbox{-a.s.}
\]
\end{proposition}

We need to introduce an additional notation. In what follows, for any
$\eta\in (\eta_{\mathrm{c}},\bar{\eta}_{\mathrm{c}})\setminus\{1\}$, we fix an
arbitrary small $\gamma=\gamma(\eta)\in(0,\frac{10^{-2}\eta_{\mathrm{c}}}{\alpha})$ such that
\[
\gamma < \cases{\ds \frac{10^{-2}}{\alpha}\min\{1-\eta,\eta\}, &\quad $\mbox{if } \eta
< 1$,\vspace*{2pt}
\cr
\ds\frac{10^{-2}}{\alpha}\min\{\eta-1, 2-\eta\}, &\quad $\mbox{if } \eta> 1$,}
\]
and define
\begin{eqnarray*}
S_\eta &:=& \bigl\{x\in(0,1)\dvtx \mbox{there exists a sequence }
(s_n,y_n)\to(t,x)
\\
&& \hspace*{10pt} \mbox{with }\bDelta X_{s_n} \bigl(
\{y_n\} \bigr)\geq (t-s_n)^{{1}/({1+\beta})-\gamma}|x-y_n|^{\eta-\eta_c}
\bigr\}.
\end{eqnarray*}
To prove the above proposition we have to verify the following two lemmas.

\begin{lemma}
\label{P1}
For every $\eta\in(\eta_c,\overline{\eta}_c)\setminus\{1\}$, we have
\[
\mathbf{P} \bigl(H_{Z^2}(x)\geq\eta-2\alpha\gamma\mbox{ for all }x\in
(0,1)\setminus S_{\eta} \bigr)=1.
\]
\end{lemma}

\begin{lemma}
\label{P2}
For every $\eta\in(\eta_c,\overline{\eta}_c)\setminus\{1\}$, we have
\[
\operatorname{dim}(S_{\eta})\leq (1+\beta) (\eta-\eta_c),
\qquad \mathbf{P}\mbox{-a.s.}
\]
\end{lemma}

With Lemmas~\ref{P1} and \ref{P2} in hand, we immediately get the following.

\begin{pf*}{Proof of Proposition~\ref{prop0802}}
It follows easily from Lemma~\ref{P1} that  $\tilde{\mathcal E}_{Z^2,\eta}\subset S_{\eta+2\alpha\gamma+\varepsilon}$
for every $\varepsilon>0$ and every $\eta\in(\eta_c,\overline{\eta}_c)\setminus\{1\}$. Therefore,
\[
\operatorname{dim}(\tilde{\mathcal E}_{Z^2,\eta})\leq \lim
_{\varepsilon\to0}\operatorname{dim}(S_{\eta+2\alpha\gamma+\varepsilon}).
\]
Using Lemma~\ref{P2}, we then get
\[
\operatorname{dim}(\tilde{\mathcal E}_{Z^2,\eta})\leq (1+\beta) (\eta+2
\alpha\gamma-\eta_c),\qquad \mathbf{P}\mbox{-a.s.}
\]
Since $\gamma$ can be chosen arbitrary small, the result for $\eta\neq1$ follows immediately.
The inequality for $\eta=1$ follows from the monotonicity in $\eta$ of the sets $\tilde{\mathcal E}_{Z^2,\eta}$.
\end{pf*}

Let $\varepsilon\in (0,\eta_{\mathrm{c}}/2)$ be arbitrarily small.
We introduce a ``good'' event  $A^{\varepsilon}$ which will be
frequently used throughout the proofs. On this event, with high
probability, $V$ from~(\ref{eq08021}) is bounded by a constant,
and there is a bound on the sizes of jumps. By Lemma~2.14
of~\cite{FMW10}, there exists a constant
$C_{(\fontsize{8.36}{10.36}{\selectfont\ref{eq08022}})}=C_{(\fontsize{8.36}{10.36}{\selectfont\ref{eq08022}})}(\varepsilon,\gamma)$
such that
%
\begin{equation}
\label{eq08022} \mathbf{P} \bigl(|\bDelta X_s|> C_{(\fontsize{8.36}{10.36}{\selectfont\ref{eq08022}})}
(t-s)^{(1+\beta)^{-1}-\gamma} \mbox{ for some } s<t \bigr)\leq \varepsilon/3.
\end{equation}
Then we fix another constant $C_{(\fontsize{8.36}{10.36}{\selectfont\ref{eq08023}})}=C_{(\fontsize{8.36}{10.36}{\selectfont\ref{eq08023}})}(\varepsilon,\gamma)$ such
that
%
\begin{equation}
\label{eq08023} \mathbf{P}(V\leq C_{(\fontsize{8.36}{10.36}{\selectfont\ref{eq08023}})})\geq 1-\varepsilon/3.
\end{equation}
Recall that, by Theorem~1.2 in~\cite{FMW10}, $x\mapsto X_t(x)$ is $\mathbf{P}$-a.s.
H\"older continuous with any exponent less than $\eta_{\mathrm{c}}$. Hence, we can define  a constant
$C_{(\fontsize{8.36}{10.36}{\selectfont\ref{eq080220}})}=C_{(\fontsize{8.36}{10.36}{\selectfont\ref{eq080220}})}(\varepsilon)$ such that\vspace*{-3pt}
%
\begin{equation}
\label{eq080220}
\mathbf{P} \biggl(\sup_{x_1,x_2\in (0,1), x_1\neq x_2}
\frac{|X_t(x_1)-X_t(x_2)|}{|x_1-x_2|^{\eta_{\mathrm{c}}-\varepsilon}}\leq C_{(\fontsize{8.36}{10.36}{\selectfont\ref{eq080220}})} \biggr)\geq 1-\varepsilon/3.
\end{equation}
Now\vspace*{-3pt} we are ready to define
%
\begin{eqnarray}
\hspace*{12pt} A^{\varepsilon}&:= & \bigl\{|\bDelta X_s|\leq C_{(\fontsize{8.36}{10.36}{\selectfont\ref{eq08022}})}
(t-s)^{(1+\beta)^{-1}-\gamma} \mbox{ for all } s<t \bigr\}
\nonumber
\\[-10pt]
\label{goodset}
\\[-10pt]
\nonumber
&&{}\cap \{ V\leq C_{(\fontsize{8.36}{10.36}{\selectfont\ref{eq08023}})}\}\cap \biggl\{\sup
_{x_1,x_2\in (0,1), x_1\neq x_2} \frac{|X_t(x_1)-X_t(x_2)|}{|x_1-x_2|^{\eta_{\mathrm{c}}-\varepsilon}}\leq C_{(\fontsize{8.36}{10.36}{\selectfont\ref{eq080220}})} \biggr\}.
\end{eqnarray}
Clearly, by~(\ref{eq08022}), (\ref{eq08023}) and (\ref{eq080220}),
$\mathbf{P}(A^{\varepsilon})\geq 1-\varepsilon$. See (3.4)
in~\cite{FMW10} for the analogous definition.

The proof of Lemma~\ref{P2} is rather short, so we will give it first.

\begin{pf*}{Proof of Lemma~\ref{P2}}
To every jump $(s,y,r)$ of the measure $\cN$ (in what follows in the paper we will
usually call them
simply ``jumps'') with
\[
(s,y,r)\in D_{j,n}:=\bigl[t-2^{-j}, t-2^{-j-1}\bigr)
\times(0,1)\times\bigl[2^{-n-1}, 2^{-n}\bigr)
\]
we assign the\vspace*{-5pt} ball
%
\begin{equation}
\label{balls} B^{(s,y,r)}:=B \biggl(y, \biggl(\frac{2^{-n}}{(2^{-j-1})^{{1}/({1+\beta})-\gamma}}
\biggr)^{1/(\eta-\eta_c)} \biggr).
\end{equation}
We used here the obvious notation $B(y,\delta)$ for the ball in $\R$ with the center at $y$ and radius $\delta$. Define
 $n_0(j):=j[\frac{1}{1+\beta}-\frac{\gamma}{4}]$. It follows from \eqref{eq08022} and \eqref{goodset} that, on
$A^\varepsilon$, there are no jumps bigger than $2^{-n_0(j)}$ in the time interval $[t-2^{-j}, t-2^{-j-1})$.

It is easy to see that every point from $S_\eta$ is contained in infinitely many balls $B^{(s,y,r)}$.
Therefore, for every $J\geq1$, the\vspace*{-3pt} set
\[
\bigcup_{j\geq J, n\geq 1}\bigcup_{(s,y,r)\in D_{j,n}}
B^{(s,y,r)}
\]
covers $S_\eta$. From \eqref{eq08022} and \eqref{goodset}, we conclude
that, on $A^\varepsilon$, there are no jumps bigger than
$C_{\fontsize{8.36}{10.36}{\selectfont\eqref{eq08022}}}2^{-(j+1)({1}/({1+\beta})-\gamma)}$ in the time
interval $s\in [t-2^{-j},t-2^{-j-1})$ for any $j\geq1$. Define
$n_0(j):=j[\frac{1}{1+\beta}-\frac{\gamma}{4}]$. Clearly, there exists $J_0$
such that for all $j\geq J_0$ there are no jumps bigger than $2^{-n_0(j)}$
in the time interval $[t-2^{-j}, t-2^{-j-1})$. Hence, for every $J\geq J_0$, the
set
\[
S_\eta(J):=\bigcup_{j\geq J, n\geq n_0(j)}\bigcup
_{(s,y,r)\in D_{j,n}} B^{(s,y,r)}
\]
covers $S_\eta$ for every $\omega\in A^\varepsilon$.

It follows from the formula for the compensator that, on the event\break
$ \{\sup_{s\leq t}X_s((0, 1))\leq N \}$,
the intensity of jumps with
$(s,y,r)\in D_{j,n}$ is bounded by
\[
N2^{-j-1}\int_{2^{-n-1}}^{2^{-n}}\varrho
r^{-2-\beta}\,dr=\frac{N\varrho(2^{1+\beta}-1)}{2(1+\beta)}2^{n(1+\beta)-j}=:\lambda_{j,n}.
\]
Therefore, the intensity of jumps with $(s,y,r)\in\bigcup_{n=n_0(j)}^{n_1(j)}D_{j,n}=:\tilde{D}_j$, where
$n_1(j)=j[\frac{1}{1+\beta}+\frac{\gamma}{4}]$, is bounded by
\[
\sum_{n=n_0(j)}^{n_1(j)}\lambda_{j,n}\leq
\frac{N\varrho 2^\beta}{(\beta+1)}2^{j(1+\beta)\gamma/4}=:\Lambda_j.
\]
The number of such jumps does not exceed $2\Lambda_j$ with the probability $1-e^{-(1-2\log2)\Lambda_j}$.
This is immediate from the exponential Chebyshev inequality applied to Poisson distributed random variables.
Analogously, the number of
jumps with $(s,y,r)\in D_{j,n}$ does not exceed $2\lambda_{j,n}$ with the probability
at least $1-e^{-(1-2\log2)\lambda_{j,n}}$.
Since
\[
\sum_j \Biggl(e^{-(1-2\log2)\Lambda_j}+\sum
_{n=n_1(j)}^\infty e^{-(1-2\log2)\lambda_{j,n}} \Biggr)<\infty,
\]
we conclude, applying the Borel--Cantelli lemma that, for almost every $\omega$ from the set
$A^\varepsilon\cap \{\sup_{s\leq t}X_s((0,1))\leq N \}$,
there exists $J(\omega)$ such that for all $j\geq J(\omega)$ and $n\geq n_1(j)$,
the numbers of jumps in $\tilde{D}_j$
and in $D_{j,n}$ are bounded by $2\Lambda_j$ and $2\lambda_{j,n}$, respectively.

The radius of every ball corresponding to the jump in $\tilde{D}_j$ is bounded by
$r_j:=C2^{-({3\gamma})/({4(\eta-\eta_c)})j}$. Thus, one can easily see that
\[
\sum_{j=1}^\infty \Biggl(2
\Lambda_j r_j^\theta+\sum
_{n=n_1(j)}^\infty 2\lambda_{j,n} \biggl(
\frac{2^{-n}}{(2^{-j-1})^{{1}/({1+\beta})-\gamma}} \biggr)^{\theta/(\eta-\eta_c)} \Biggr)<\infty
\]
for every $\theta>(1+\beta)(\eta-\eta_c)$. This yields the desired bound for the Hausdorff dimension
for almost every $\omega\in A^\varepsilon\cap \{\sup_{s\leq t}X_s((0,1))\leq N \}$.
Letting $N\to\infty$ and $\varepsilon\to0$ completes the proof.
\end{pf*}

The remaining part of this section  will be devoted to the proof of Lemma~\ref{P1}.

Since $S_\eta=\bigcap_{J\geq1}S_\eta(J)$,
\begin{eqnarray*}
&& \bigl\{H_{Z^2}(x)\geq \eta-2\alpha\gamma,\forall x\in (0,1)
\setminus S_\eta \bigr\}
\\[-3pt]
&& \qquad =\bigcap_{J\geq1} \bigl\{H_{Z^2}(x)
\geq \eta-2\alpha\gamma,\forall x\in (0,1)\setminus S_\eta(J) \bigr\}.
\end{eqnarray*}
Thus, it suffices to show that
%
\begin{equation}
\label{reduction} \mathbf{P} \bigl(H_{Z^2}(x)\geq \eta-2\alpha\gamma,
\forall x\in (0,1)\setminus S_\eta(J) \bigr)=1
\end{equation}
for every $J\geq1$.

Before we start proving~(\ref{reduction}), let us introduce some further notation. For any $x_1,x_2\in \R,
\eta\in (\eta_{\mathrm{c}},\bar\eta_{\mathrm{c}})$, define
\begin{eqnarray*}
\tilde{p}^{\alpha,\eta}_s(x,y) &:=& \cases{\ds p_{s}^{\alpha}(x)-p_{s}^{\alpha}(y),
& \quad \mbox{if }$\eta\leq 1, s>0$,
\cr
\ds p_{s}^{\alpha}(x)-p_{s}^{\alpha}(y)
- (x-y)\frac{\partial p^{\alpha}_s(y)}{\partial y},& \quad \mbox{if }$\eta\in (1,\bar\eta_{\mathrm{c}}),
s>0$,}
\\
\tilde{p}^{\alpha,\eta,\prime}_s(x,y) &:=& \frac{\partial p^{\alpha}_s(x)}{\partial x}-
\frac{\partial p^{\alpha}_s(y)}{\partial y}, \qquad \eta\in (1,\bar\eta_{\mathrm{c}}),
\\
\tilde{Z}^{2,\eta}_s(x_1, x_2)
&:=& \int_0^s\!\int_{\mathsf{R}} M
\bigl(d(u,y) \bigr) \tilde{p}^{\alpha,\eta}_{t-u}(x_1-y,x_2-y),
\qquad s\in [0,t],
\\
\bDelta\tilde{Z}^{2,\eta}_s(x_1,
x_2)&:=& \tilde{Z}^{2,\eta}_s(x_1,
x_2)-\tilde{Z}^{2,\eta}_{s-}(x_1,
x_2), \qquad s\in (0,t],
\\
\tilde{Z}^{2,\eta,\prime}_s(x_1, x_2)&:=&
\int_0^s\!\int_{\mathsf{R}} M
\bigl(d(u,y) \bigr) \tilde{p}^{\alpha,\eta,\prime}_{t-u}(x_1-y,x_2-y),
\qquad s\in [0,t],
\\
\bDelta\tilde{Z}^{2,\eta,\prime}_s(x_1,
x_2)&:=& \tilde{Z}^{2,\eta,\prime}_s(x_1,
x_2)-\tilde{Z}^{2,\eta,\prime}_{s-}(x_1,
x_2), \qquad\eta\in (1,\bar\eta_{\mathrm{c}}),  s\in (0,t].
\end{eqnarray*}
Also for any $N, J\geq 1$, let
\begin{eqnarray*}
\tilde{S}_{\eta}(N,J) &:=& \bigl\{(x_1,x_2)
\in \R^2\dvtx \exists x_0\in (0,1)\setminus
S_{\eta}(J)
\\
&& \hspace*{3pt} \qquad  \mbox{such that } x_1,x_2\in B
\bigl(x_0, 2^{-N} \bigr) \bigr\}
\end{eqnarray*}
and
\begin{eqnarray*}
\tilde{S}'_{\eta}(J)& =& \bigl\{(x_1,x_2)
\in \R^2\dvtx \exists x_0\in (0,1)\setminus
S_{\eta}(J)
\\
&&  \hspace*{6pt} \mbox{such that } x_1,x_2\in B
\bigl(x_0, 4|x_1-x_2| \bigr) \bigr\}.
\end{eqnarray*}

We split the proof of (\ref{reduction}) into several steps.

\begin{lemma}
\label{Step1}
Fix arbitrary (deterministic) $x_1, x_2\in \R$, and $ \eta\in (\eta_{\mathrm{c}},\bar\eta_{\mathrm{c}})$.
Then for any $N,J\geq 1$, there exists a constant $C_{(\fontsize{8.36}{10.36}{\selectfont\ref{eq07031}})}(J)\geq 1$ such that
\begin{eqnarray}
&& \bigl\llvert \bDelta\tilde{Z}^{2,\eta}_s(x_1,
x_2) \bigr\rrvert {\mathsf 1}_{A^\varepsilon} {\mathsf
1}_{ \{(x_1,x_2)\in\tilde{S}_{\eta}(N,J) \}}
\nonumber
\\[-8pt]
\label{eq07031}
\\[-8pt]
\nonumber
&&\qquad \leq C_{(\fontsize{8.36}{10.36}{\selectfont\ref{eq07031}})}(J)|x_1-x_2|^{\eta_c-\alpha\gamma}2^{-N(\eta-\eta_c)}
\qquad \forall s\leq t.
\end{eqnarray}
\end{lemma}

\begin{pf}
Let $(y,s,r)$ be the point of an arbitrary jump of the measure $\cN$ with $s\leq t$. Then for the corresponding jump
of $\tilde{Z}^{2,\eta}_s(x_1, x_2)$ we get the following bound:
%
\begin{equation}
\label{eq10101} \bigl\llvert \bDelta\tilde{Z}^{2,\eta}_s(x_1,
x_2) \bigr\rrvert \leq r \bigl\llvert \tilde{p}^{\alpha,\eta}_{t-s}(x_1-y,x_2-y)
\bigr\rrvert.
\end{equation}
Now on the event $ \{(x_1,x_2)\in \tilde{S}_{\eta}(N,J) \}$,
there exists a point $x_0\in (0,1)\setminus S_\eta(J)$ such that
$x_1,x_2\in  B (x_0, 2^{-N} )$,  and for
$s\geq t-2^{-J}$ we have
\[
r\leq (t-s)^{{1}/({1+\beta})-\gamma}|y-x_0|^{\eta-\eta_c}.
\]
This and~(\ref{eq10101}) imply that  for $s\geq t-2^{-J}$
\begin{eqnarray}
I &=& I(s,y,x_1,x_2):= \bigl\llvert \bDelta
\tilde{Z}^{2,\eta}_s(x_1, x_2) \bigr
\rrvert {\mathsf 1}_{A^\varepsilon}{\mathsf 1}_{ \{(x_1,x_2)\in \tilde{S}_{\eta}(N,J) \}}
\nonumber
\\[-8pt]
\label{eq10102}
\\[-8pt]
\nonumber
&\leq& (t-s)^{{1}/({1+\beta})-\gamma}|y-x_0|^{\eta-\eta_c} \bigl
\llvert \tilde{p}^{\alpha,\eta}_{t-s}(x_1-y,x_2-y)
\bigr\rrvert.
\end{eqnarray}
Applying\vspace*{1pt} Lemma~\ref{L1old} (if $\eta\leq 1$) or Corollary~\ref{cor14101} (if $\eta>1$)
 with $\delta=\eta-\alpha\gamma$ to $\tilde{p}^{\alpha,\eta}_{t-s}(x_1-y,x_2-y)$, we conclude from~(\ref{eq10102}) that, for
$s\geq t-2^{-J}$,
\begin{eqnarray}
I &\leq& (t-s)^{-(\eta-\eta_c)/\alpha}|y-x_0|^{\eta-\eta_c}|x_1-x_2|^{\eta-\alpha\gamma}\nonumber
\\
&&{}\times \biggl(p^\alpha_{1} \biggl(
\frac{x_1-y}{2(t-s)^{1/\alpha}} \biggr)+p^\alpha_{1} \biggl(
\frac{x_2-y}{2(t-s)^{1/\alpha}} \biggr) \biggr)
\nonumber
\\[-8pt]
\label{eq11102}
\\[-8pt]
\nonumber
&\leq& C_{(\fontsize{8.36}{10.36}{\selectfont\ref{eq11102}})} (t-s)^{-(\eta-\eta_c)/\alpha}|y-x_0|^{\eta-\eta_c}|x_1-x_2|^{\eta-\alpha\gamma}\nonumber
\\
&&{}\times \biggl(\frac{|x_1-y|+|x_2-y|}{(t-s)^{1/\alpha}}\vee 1 \biggr)^{-\alpha-1},\nonumber
\end{eqnarray}
where the last inequality follows from the standard bound
%
\begin{equation}
\label{eq326} p^{\alpha}_1(z)\leq C_{(\fontsize{8.36}{10.36}{\selectfont\ref{eq326}})} \bigl(|z|\vee
1\bigr)^{-\alpha-1}, \qquad z\in \R.
\end{equation}


One can easily check by separating the cases $|x_1-y|+|x_2-y|<(t-s)^{1/\alpha}$ and $|x_1-y|+|x_2-y|\geq (t-s)^{1/\alpha}$ that
\[
|x_1-x_2|^{\eta-\eta_c} \biggl(\frac{|x_1-y|+|x_2-y|}{(t-s)^{1/\alpha}}\vee 1
\biggr)^{-\alpha-1}\leq (t-s)^{(\eta-\eta_c)/\alpha},
\]
and hence
\begin{equation}
\label{eq10103} I \leq C_{(\fontsize{8.36}{10.36}{\selectfont\ref{eq11102}})} |x_1-x_2|^{\eta_c-\alpha\gamma}|y-x_0|^{\eta-\eta_c},
\qquad s\geq t-2^{-J}.
\end{equation}
If $|y-x_0|\leq 2^{-N+1}$, then we obtain the bound
%
\begin{equation}
\label{Step1.1} I\leq 2C_{(\fontsize{8.36}{10.36}{\selectfont\ref{eq11102}})}|x_1-x_2|^{\eta_c-\alpha\gamma}2^{-N(\eta-\eta_c)},
\qquad s\geq t-2^{-J}.
\end{equation}

Now consider the case $|y-x_0|> 2^{-N+1}$. Here, we
treat separately two subcases:
$|y-x_0|\leq (t-s)^{1/\alpha}$ and $|y-x_0|> (t-s)^{1/\alpha}$. First, if
$|y-x_0|\leq (t-s)^{1/\alpha}$,
then it follows from~(\ref{eq11102})
that
%
\begin{eqnarray}
\nonumber
I &\leq & C_{(\fontsize{8.36}{10.36}{\selectfont\ref{eq11102}})} (t-s)^{-(\eta-\eta_c)/\alpha}|y-x_0|^{\eta-\eta_c}|x_1-x_2|^{\eta-\alpha\gamma}
\\
\label{Step1.2} &\leq& C_{(\fontsize{8.36}{10.36}{\selectfont\ref{eq11102}})} |x_1-x_2|^{\eta-\alpha\gamma}
\\
\nonumber
&\leq & C_{(\fontsize{8.36}{10.36}{\selectfont\ref{eq11102}})} |x_1-x_2|^{\eta_c-\alpha\gamma}2^{-(N-1)(\eta-\eta_c)},
\qquad s\geq t-2^{-J}.
\end{eqnarray}
Second, if $|y-x_0|> (t-s)^{1/\alpha}$, then we recall that $|y-x_0|> (t-s)^{1/\alpha}\vee 2^{-N+1}$ and $|x_i-x_0|\leq
2^{-N}, i=1,2$, to get that
%
\begin{equation}
|x_i-y| \geq |x_0-y|/2, \qquad i=1,2.
\end{equation}
Combining this with (\ref{eq11102}), we obtain
\begin{eqnarray}
\hspace*{20pt} I&\leq & C_{(\fontsize{8.36}{10.36}{\selectfont\ref{eq11102}})} (t-s)^{-(\eta-\eta_c)/\alpha}|y-x_0|^{\eta-\eta_c}|x_1-x_2|^{\eta-\alpha\gamma}
\biggl(\frac{|x_0-y|}{(t-s)^{1/\alpha}} \biggr)^{-\alpha-1}
\nonumber
\\
\nonumber
&\leq & C_{(\fontsize{8.36}{10.36}{\selectfont\ref{eq11102}})} |x_1-x_2|^{\eta-\alpha\gamma}
\biggl(\frac{|x_0-y|}{(t-s)^{1/\alpha}} \biggr)^{\eta-\eta_c-\alpha-1}
\nonumber
\\[-8pt]
\label{Step1.3}
\\[-8pt]
\nonumber
&\leq & C_{(\fontsize{8.36}{10.36}{\selectfont\ref{eq11102}})} |x_1-x_2|^{\eta-\alpha\gamma}
\\
\nonumber
&\leq & C_{(\fontsize{8.36}{10.36}{\selectfont\ref{eq11102}})} |x_1-x_2|^{\eta_c-\alpha\gamma}2^{-(N-1)(\eta-\eta_c)},
\qquad s\geq t-2^{-J}.
\end{eqnarray}

Finally, we consider the jumps $(y,s,r)$ with $s<t-2^{-J}$. On the event
$A^\varepsilon$,
\[
r\leq (t-s)^{{1}/({1+\beta})-\gamma}.
\]
Using Lemma~\ref{L1old} (or  Corollary~\ref{cor14101}) with $\delta=\eta-\alpha\gamma$ once again, we see
from~(\ref{eq10101}) that
\begin{eqnarray}
\nonumber
\hspace*{20pt} I &\leq & C_{(\fontsize{8.36}{10.36}{\selectfont\ref{Step1.4}})}|x_1-x_2|^{\eta-\alpha\gamma}(t-s)^{-(\eta-\eta_c)/\alpha}
\nonumber
\\
\label{Step1.4}&\leq & C_{(\fontsize{8.36}{10.36}{\selectfont\ref{Step1.4}})} 2^{J(\eta-\eta_c)/\alpha}|x_1-x_2|^{\eta-\alpha\gamma}
\\
&\leq & C_{(\fontsize{8.36}{10.36}{\selectfont\ref{Step1.4}})} 2^{J(\eta-\eta_c)/\alpha}|x_1-x_2|^{\eta_c-\alpha\gamma}2^{-(N-1)(\eta-\eta_c)},
\qquad s< t-2^{-N}.
\nonumber
\end{eqnarray}
Combining (\ref{Step1.1})--(\ref{Step1.4}), we get the desired result.
\end{pf}

By a similar argument, we can get the following result.

\begin{lemma}
\label{lemStep5}
Let $\bar\eta_c>1$.
Fix arbitrary (deterministic) $x_1, x_2\in \R$, and $ \eta\in (1,\bar\eta_{\mathrm{c}})$. Then for any $J\geq 1$, there exists a constant
$C_{(\fontsize{8.36}{10.36}{\selectfont\ref{eq07032}})}(J)$ such that
\begin{eqnarray}
&& \bigl\llvert \bDelta\tilde{Z}^{\eta,2,\prime}_s(x_1,
x_2) \bigr\rrvert {\mathsf 1}_{A^\varepsilon} {\mathsf
1}_{ \{(x_1,x_2)\in\tilde{S}'_\eta(J) \}}
\nonumber
\\[-8pt]
\label{eq07032}
\\[-8pt]
\nonumber
&& \qquad \leq C_{(\fontsize{8.36}{10.36}{\selectfont\ref{eq07032}})}(J)|x_1-x_2|^{\eta-1-\alpha\gamma}
\qquad \forall s\leq t.
\end{eqnarray}
\end{lemma}

Having an upper bound for absolute values of the jumps of $\tilde{Z}^2(x_1, x_2)$, we can give some
estimate for  $\tilde{Z}^2_t(x_1, x_2)$ itself.

\begin{lemma}
\label{Step2}
Fix arbitrary (deterministic) $x_1,x_2\in (0,1)$, and $\eta\in (\eta_{\mathrm{c}},\bar\eta_{\mathrm{c}})$.
\begin{longlist}[(a)]
\item[(a)]
Then there exists a constant $C_{(\fontsize{8.36}{10.36}{\selectfont\ref{eq07034}})}$, such that
for any $N,J\geq 1$,
%
\begin{eqnarray}
&& \mathbf{P} \bigl(\bigl|\tilde{Z}^{2,\eta}_t(x_1,
x_2)\bigr|\geq 2C_{(\fontsize{8.36}{10.36}{\selectfont\ref{eq07031}})}(J)|x_1-x_2|^{\eta_c-2\alpha\gamma}
2^{-N(\eta-\eta_c)}, A^\varepsilon,\nonumber \\
&&{}\hspace*{168pt}\quad (x_1,x_2)\in
\tilde{S}_\eta(N,J) \bigr)
\nonumber
\\[-8pt]
\label{eq07034}
\\[-8pt]
&& \qquad \leq \bigl(C_{(\fontsize{8.36}{10.36}{\selectfont\ref{eq07034}})}2^{-\alpha\gamma N} \bigr)^{|x_1-x_2|^{-\alpha\gamma}}.
\nonumber
\end{eqnarray}
\item[(b)] Let $\bar\eta_c>1$ and assume $\eta\in (1,\bar\eta_{\mathrm{c}})$.
Then there exists a constant $C_{(\fontsize{8.36}{10.36}{\selectfont\ref{eq07035}})}$, such that
for any $J\geq 1$,
\begin{eqnarray}
\qquad&&\mathbf{P} \bigl(\bigl|\tilde{Z}^{2,\eta,\prime}_t(x_1,
x_2)\bigr|\geq 2C_{(\fontsize{8.36}{10.36}{\selectfont\ref{eq07032}})}(J)|x_1-x_2|^{\eta-1-2\alpha\gamma},
A^\varepsilon, (x_1,x_2)\in\tilde{S}'_\eta(J)
\bigr)
\nonumber
\\[-8pt]
\label{eq07035} \\[-8pt]
\nonumber
&& \qquad \leq \bigl(C_{(\fontsize{8.36}{10.36}{\selectfont\ref{eq07035}})}|x_1-x_2|^{\alpha\gamma}
\bigr)^{|x_1-x_2|^{-\alpha\gamma}}.
\end{eqnarray}
\end{longlist}
\end{lemma}

\begin{pf}
(a) According to Lemma~2.15 from \cite{FMW10}, there exist spectrally positive $(1+\beta)$-stable processes $L^+$ and $L^-$
such that
%
\begin{equation}
\label{eq17101} \tilde{Z}^{2,\eta}_s(x_1,
x_2)=L^+_{T^{\eta}_+(s)}-L^-_{T^{\eta}_-(s)},\qquad s\leq t,
\end{equation}
where
%
\begin{equation}
\label{eq17102}
\qquad T^{\eta}_\pm(s) := \int_0^s
du \int_{\mathbb{R}}X_u(dy) \bigl( \bigl(
\tilde{p}^{\alpha,\eta}_{t-u}(x_1-y,x_2-y)
\bigr)^{\pm} \bigr)^{1+\beta},\qquad 0\leq s\leq t.\hspace*{-14pt}
\end{equation}
(Note that $L^+, L^-$ also depend on $\eta$;  however, we omit the corresponding superindex $\eta$ to simplify the notation.)
Therefore, we get
%
\begin{eqnarray}
\nonumber
&& \mathbf{P} \bigl(\bigl|\tilde{Z}^{2,\eta}_t(x_1,
x_2)\bigr|\geq 2C_{(\fontsize{8.36}{10.36}{\selectfont\ref{eq07031}})}(J)|x_1-x_2|^{\eta_c-2\alpha\gamma}2^{-N(\eta-\eta_c)},
A^\varepsilon,\\[-3pt]
\nonumber
&&\hspace*{91pt} \qquad\qquad\qquad\qquad \bigl(x_1,x_2\in\tilde{S}_\eta(N,J)
\bigr) \bigr)
\\
&& \qquad  \leq \mathbf{P} \bigl(\bigl|L^+_{T^{\eta}_+(s)}\bigr|\geq
C_{(\fontsize{8.36}{10.36}{\selectfont\ref{eq07031}})}(J)|x_1-x_2|^{\eta_c-2\alpha\gamma}2^{-N(\eta-\eta_c)},
A^\varepsilon,
\nonumber
\\[-9pt]
\label{eq10104}\\[-12pt]
\nonumber
&&\hspace*{92pt} \qquad\qquad\qquad\qquad  \bigl(x_1,x_2\in\tilde{S}_\eta(N,J)
\bigr) \bigr)
\\
\nonumber
&&\qquad \quad {}+ \mathbf{P} \bigl(\bigl|L^-_{T^{\eta}_{-}(s)}\bigr|\geq
C_{(\fontsize{8.36}{10.36}{\selectfont\ref{eq07031}})}(J)|x_1-x_2|^{\eta_c-2\alpha\gamma}2^{-N(\eta-\eta_c)},
A^\varepsilon,\\[-3pt]
&& \nonumber \hspace*{106pt} \qquad\qquad\qquad\qquad\bigl(x_1,x_2\in\tilde{S}_\eta(N,J)
\bigr) \bigr).
\end{eqnarray}
By going through the derivation of (3.43) in~\cite{FMW11}, one can easily get that in our setting,
on the event $A^\varepsilon$, for any $\eta\in (\eta_{\mathrm{c}},\bar\eta_{\mathrm{c}})$ and any
$\varepsilon_{1}\in (0,\alpha\beta\gamma)$, there exists a~constant
 $C_{(\fontsize{8.36}{10.36}{\selectfont\ref{eqtime1}})}=C_{(\fontsize{8.36}{10.36}{\selectfont\ref{eqtime1}})}(\varepsilon,\varepsilon_1,\eta)$ such that
%
\begin{equation}
\label{eqtime1} T^{\eta}_{\pm}(t) \leq C_{(\fontsize{8.36}{10.36}{\selectfont\ref{eqtime1}})}|x_1-x_2|^{\alpha-\beta-\varepsilon_{1}}=:
\hat{T}(x_1,x_2).
\end{equation}
From this bound, Lemmas \ref{L3} and \ref{Step1} we get
\begin{eqnarray*}
&&\mathbf{P} \bigl(\bigl|L^\pm_{T^{\eta}_{\pm}(t)}\bigr|\geq C_{(\fontsize{8.36}{10.36}{\selectfont\ref{eq07031}})}(J)|x_1-x_2|^{\eta_c-2\alpha\gamma}2^{-N(\eta-\eta_c)},A^\varepsilon,
\bigl(x_1,x_2\in\tilde{S}_\eta(N,J) \bigr)
\bigr)
\\[-2pt]
&& \qquad \leq \mathbf{P} \Bigl(\bigl|L^\pm_{T^{\eta}_{\pm}(t)}\bigr|\geq
C_{(\fontsize{8.36}{10.36}{\selectfont\ref{eq07031}})}(J)|x_1-x_2|^{\eta_c-2\alpha\gamma}2^{-N(\eta-\eta_c)},A^\varepsilon,
\\[-2pt]
&& \hspace*{25pt}\qquad \quad \sup_{s\leq T^{\eta}_{\pm}} \bDelta L^\pm_s
\leq C_{(\fontsize{8.36}{10.36}{\selectfont\ref{eq07031}})}(J)|x_1-x_2|^{\eta_c-\alpha\gamma}2^{-N(\eta-\eta_c)}
\Bigr)
\\[-2pt]
&&\qquad\leq \mathbf{P} \Bigl(\sup_{0\leq s\leq \hat{T}(x_1,x_2)}\bigl|L^\pm_{s}\bigr|
\mathsf{1} \Bigl\{\sup_{0\leq v\leq
s}\bDelta L^\pm_{v}
\leq C_{(\fontsize{8.36}{10.36}{\selectfont\ref{eq07031}})}(J)|x_1-x_2|^{\eta_c-\alpha\gamma}2^{-N(\eta-\eta_c)}
\Bigr\}
\\[-2pt]
&& \hspace*{148pt}\qquad\quad \geq C_{(\fontsize{8.36}{10.36}{\selectfont\ref{eq07031}})}(J)|x_1-x_2|^{\eta_c-2\alpha\gamma}2^{-N(\eta-\eta_c)}
\Bigr)
\\
&& \qquad \leq \biggl(\frac{ C_{(\fontsize{8.36}{10.36}{\selectfont\ref{eq07033}})} C_{(\fontsize{8.36}{10.36}{\selectfont\ref{eqtime1}})}|x_1-x_2|^{\alpha-\beta-\varepsilon_1}}{|x_1-x_2|^{-\alpha\gamma}
 (C_{(\fontsize{8.36}{10.36}{\selectfont\ref{eq07031}})}(J)|x_1-x_2|^{\eta_c-\alpha\gamma}2^{-N(\eta-\eta_c)})^{1+\beta} } \biggr)^{|x_1-x_2|^{-\alpha\gamma}}
\\
&& \qquad \quad {}+\exp \biggl\{-\frac{(C_{(\fontsize{8.36}{10.36}{\selectfont\ref{eq07031}})}(J)|x_1-x_2|^{\eta_c-2\alpha\gamma}2^{-N(\eta-\eta_c)})^{(1+\beta)/\beta}}{
C_{(\fontsize{8.36}{10.36}{\selectfont\ref{eq07033}})} (C_{(\fontsize{8.36}{10.36}{\selectfont\ref{eqtime1}})}|x_1-x_2|^{\alpha-\beta-\varepsilon_{1}})^{1/\beta}} \biggr\}
\\
&& \qquad = \biggl(C\frac{|x_1-x_2|^{\alpha\gamma(2+\beta)-\varepsilon_1+1}}{
2^{-N(\eta-\eta_c)(1+\beta)} } \biggr)^{|x_1-x_2|^{-\alpha\gamma}}
\\
&&\qquad \quad {}+\exp \bigl\{-c|x_1-x_2|^{-1/\beta+\varepsilon_1/\beta-2\alpha\gamma(1+\beta)/\beta}2^{-N(\eta-\eta_c)(1+\beta)/\beta}
\bigr\}
\\
&& \qquad =O \bigl( \bigl(2^{-\alpha\gamma N} \bigr)^{|x_1-x_2|^{-\alpha\gamma}} \bigr),
\end{eqnarray*}
where the last equality follows since
$(\eta-\eta_c)(1+\beta)\leq 1, \varepsilon_1\leq \alpha\beta\gamma$, $|x_1-x_2|
\leq 2^{-N+1}$.
(We omit here some elementary arithmetic calculations.)

The claim follows now from~(\ref{eq10104}).

(b) The proof goes along the similar lines.
\end{pf}

\begin{lemma}
\label{lemStep6}
Let $J\geq 1$, $\bar\eta_{\mathrm{c}}>1$, $\eta\in (1,\bar\eta_{\mathrm{c}})$.
For almost every $\omega\in A^\varepsilon$, there exists $N_2=N_2(\omega)$ such that for
all $n\geq  N_2$, and
\[
(x_1,x_2)\in\tilde{S}'_\eta(J)
\cap \bigl\{ \bigl(i2^{-n},j2^{-n} \bigr), i,j\in\Z, |i-j|=1
\bigr\}
\]
the following holds:
\[
\bigl|\tilde{Z}^{2,\eta,\prime}_t(x_1, x_2)\bigr|
\leq 2C_{(\fontsize{8.36}{10.36}{\selectfont\ref{eq07032}})}(J)|x_1-x_2|^{\eta-1-2\alpha\gamma}.
\]
\end{lemma}

\begin{pf}
Define
\begin{eqnarray*}
M_{n}&:=&\max \bigl\{\bigl|\tilde{Z}^{2,\eta,\prime}_t(x_1,
x_2)\bigr|\dvtx (x_1,x_2)\in \bigl\{
\bigl(i2^{-n},j2^{-n} \bigr),i,j\in\Z, |i-j|=1 \bigr\}
\\
&&\hspace*{223pt}{}\cap (0,1)\cap \tilde{S}'_\eta(J) 
 \bigr\}.
\end{eqnarray*}
Applying Lemma~\ref{Step2}(b),
we obtain
\[
\mathbf{P} \bigl(M_{n}\geq 2C_{(\fontsize{8.36}{10.36}{\selectfont\ref{eq07032}})}(J)2^{-n(\eta-1-2\alpha\gamma)};A^\varepsilon
\bigr) \leq 2^{n} \bigl(C_{(\fontsize{8.36}{10.36}{\selectfont\ref{eq07035}})}2^{-n\alpha\gamma}
\bigr)^{2^{n\alpha\gamma}}.
\]
Let
\[
A_N:= \bigl\{M_{n}\geq 2C_{(\fontsize{8.36}{10.36}{\selectfont\ref{eq07032}})}(J)2^{-n(\eta-1-2\alpha\gamma)}
\mbox{ for some }n\geq N \bigr\}.
\]
It is clear that
\begin{eqnarray*}
\sum_{N=1}^\infty\mathbf{P}
\bigl(A_N\cap A^\varepsilon \bigr)&\leq& \sum
_{N=1}^\infty\sum_{n=N}^\infty
\mathbf{P} \bigl(M_{n}\geq 2C_{(\fontsize{8.36}{10.36}{\selectfont\ref{eq07032}})}(J)2^{-n(\eta-1-2\alpha\gamma)};A^\varepsilon
\bigr)
\\
&\leq & \sum_{N=1}^\infty\sum
_{n=N}^\infty 2^{n} \bigl(C_{(\fontsize{8.36}{10.36}{\selectfont\ref{eq07035}})}2^{-n\alpha\gamma}
\bigr)^{2^{n\alpha\gamma}}
\\
&\leq & \sum_{N=1}^\infty\sum
_{n=N}^\infty 2^{n} \bigl(C_{(\fontsize{8.36}{10.36}{\selectfont\ref{eq07035}})}
2^{-n\alpha\gamma} \bigr)^{2^{n\alpha\gamma}}
\\
&\leq & C \sum_{N=1}^\infty2^{-\alpha\gamma N}<
\infty
\end{eqnarray*}
and we are done by the Borel--Cantelli lemma.
\end{pf}

\begin{lemma}
\label{Step3}
Let $J\geq 1$, $\eta\in (\eta,\bar\eta_{\mathrm{c}})$.
For almost every $\omega\in A^\varepsilon$, there exists $N_1=N_1(\omega)$ such that for
all 
$n\geq N\geq N_1$,
\[
(x_1,x_2)\in\tilde{S}_\eta(N,J)\cap \bigl
\{ \bigl(i2^{-n},j2^{-n} \bigr), i,j\in\Z \bigr\}
\]
with $|x_1-x_2|\leq 2^{-\log^2n}$
we have the inequality
\[
\bigl|\tilde{Z}^{2,\eta}_t(x_1, x_2)\bigr|
\leq 2C_{(\fontsize{8.36}{10.36}{\selectfont\ref{eq07031}})}(J)|x_1-x_2|^{\eta_c-2\alpha\gamma}2^{-N(\eta-\eta_c)}.
\]
\end{lemma}

\begin{pf}
Define events
\[
B(x_1,x_2):= \bigl\{\tilde{Z}^{2,\eta}_t(x_1,
x_2) \geq 2C_{(\fontsize{8.36}{10.36}{\selectfont\ref{eq07031}})}(J)|x_1-x_2|^{\eta_c-2\alpha\gamma}2^{-N(\eta-\eta_c)}
\bigr\},
\]
and
\[
A_{n,N} := \mathop{\mathop{\mathop{\bigcup}_{x_1,x_2\in\{i2^{-n},i\in\Z\}}}_{\cap (0,1)\cap \tilde{S}_{\eta}(N,J)\dvtx}}_{\{|x_1-x_2|\leq 2^{-\log^2n}\}
\}}
B(x_1,x_2).
\]

Applying Lemma~\ref{Step2}(a), we obtain
\[
\mathbf{P} \bigl(A_{n,N}\cap A^\varepsilon \bigr) \leq
2^{2n} \bigl(C_{(\fontsize{8.36}{10.36}{\selectfont\ref{eq07034}})} 2^{-\alpha\gamma N} \bigr)^{2^{\alpha\gamma\log^2n}}.
\]
Let
\[
A_N:= \bigcup_{n\geq N} A_{n,N}.
\]
It is clear that
\begin{eqnarray*}
\sum_{N=1}^\infty\mathbf{P}
\bigl(A_N\cap A^\varepsilon \bigr)&\leq & \sum
_{N=1}^\infty\sum_{n=N}^\infty
\mathbf{P} \bigl(A_{n,N}\cap A^\varepsilon \bigr)
\\
&\leq & \sum_{N=1}^\infty\sum
_{n=N}^\infty 2^{2n} \bigl(
C_{(\fontsize{8.36}{10.36}{\selectfont\ref{eq07034}})} 2^{-\alpha\gamma N} \bigr)^{2^{\alpha\gamma \log^2n}}
\\
&\leq & C \sum_{N=1}^\infty2^{-\alpha\gamma N}<
\infty
\end{eqnarray*}
and we are done by the Borel--Cantelli lemma.
\end{pf}

\begin{lemma}
\label{lemStep7}
Let $J\geq 1$, $\bar\eta_{\mathrm{c}}>1$, $\eta\in (1,\bar\eta_{\mathrm{c}})$. Fix an
integer
$k_0>\max \{1+\frac{1}{\alpha-1-\beta},3\}$.
For almost every $\omega\in A^\varepsilon$
and for
all $x\in (0,1)\setminus S_\eta(J)$, there exists $V'(x)=V'(x,\omega)$ and $C_{(\fontsize{8.36}{10.36}{\selectfont\ref{eq14102}})}(J,\omega)$ such that
%
\begin{eqnarray}
&& \bigl|Z^{2}(y)-Z^{2}(x)-(y-x)V'(x)\bigr|
\nonumber
\\[-8pt]
\label{eq14102}
\\[-8pt]
\nonumber
&& \qquad \leq C_{(\fontsize{8.36}{10.36}{\selectfont\ref{eq14102}})}(J)|y-x|^{\eta-2\alpha\gamma}\qquad
\forall y\in B \bigl(x, 2^{-N_3} \bigr),
\end{eqnarray}
where
\[
N_3=\max\bigl\{N_2(\omega),N_1(\omega),
\log_2\bigl(\bigl|V'(x)\bigr|\bigr), (k_0)^{10}
\bigr\}+2
\]
and $N_2, N_1$ are from Lemmas~\ref{lemStep6} and \ref{Step3}.
\end{lemma}

\begin{remark}
$Z^{2}(x)$ [similarly for $Z^{2}(y)$] at a random point $x$ is defined via  $Z^{2}(x)=X_t(x)-Z^1(x)-Z^3(x)$, where
all the terms on the right-hand side are well defined.
\end{remark}

\begin{remark}
The lemma shows that $V'(x)$ is in fact a spatial derivative of $Z^{2}(x)$ at the point $x$.
\end{remark}

\begin{pf*}{Proof of Lemma~\ref{lemStep7}}
First, we will define $V'(y)$ for fixed points $y$.
For any $y\in \R$, let
\[
V'(y):= \int_0^t\!\int
_{\mathsf{R}} M \bigl(d(u,z) \bigr) \frac{\partial}{\partial y}
p^{\alpha}_{t-u}(y-z).
\]
Let $x$ and $\omega$ be as in the statement of the lemma.
For any $n\geq 1$,
take $x_n\in \{i2^{-n},i\in\Z\}$ satisfying the following conditions:
%
\begin{equation}
\label{eq14103} |x_n-x|\leq 2^{-n},\qquad
|x_n-x_{n+1}|=2^{-n-1} \qquad\forall n\geq 1.
\end{equation}
Applying Lemma~\ref{lemStep6}, we get for every $n\geq N_3$ the bound
\begin{eqnarray*}
\nonumber
\bigl|V'(x_n)-V'(x_{n+1})\bigr|&=&
\bigl|\tilde{Z}^{2,\eta,\prime}_t(x_n, x_{n+1})\bigr|
\\
\nonumber
&\leq & 2C_{(\fontsize{8.36}{10.36}{\selectfont\ref{eq07032}})}(J)2^{-n(\eta-1-2\alpha\gamma)}.
\end{eqnarray*}
Then for any $m>n\geq N_3$ we have
%
\begin{eqnarray}
\bigl|V'(x_n)-V'(x_{m})\bigr| &\leq &
\sum_{k=n}^{m-1}\bigl| V'(x_k)-V'(x_{k+1})\bigr|
\nonumber
\\[-8pt]
\label{11031}
\\[-8pt]
\nonumber
&\leq & C_{(\fontsize{8.36}{10.36}{\selectfont\ref{11031}})}(J)2^{-n(\eta-1-2\alpha\gamma)}.
\end{eqnarray}
This implies that $\{V'(x_n)\}_{n\geq 1}$ is a Cauchy sequence and we denote the limit by $V'(x)$. Moreover, it is
easy to check that
%
\begin{equation}
\label{eq14104} \bigl|V'(x_n)-V'(x)\bigr| \leq
C_{(\fontsize{8.36}{10.36}{\selectfont\ref{11031}})}(J)2^{-n(\eta-1-2\alpha\gamma)},\qquad n\geq N_3.
\end{equation}

Now let us check~(\ref{eq14102}).  Let
\[
y\in B \bigl(x, 2^{-N_3} \bigr)\setminus\{x\}.
\]
Then we can fix an integer $N^*\geq N_3$ such that
%
\begin{equation}
\label{13111} 2^{-N^*-1}\leq |x-y| \leq 2^{-N^*}.
\end{equation}

Fix a sequence $\{x_{n}\}_{n\geq 1}$
satisfying~(\ref{eq14103}), and $\{y_{n}\}_{n\geq 1}$ satisfying the same condition with $y$ instead of $x$.
Then for any $n\geq N_3$ we have
%
\begin{eqnarray}
\nonumber
&& \bigl|Z^{2}(y)-Z^{2}(x)-(y-x)V'(x)\bigr|
\\
\nonumber
&& \qquad \leq \bigl|Z^{2}(y_n)-Z^{2}(x_n)-(y_n-x_n)V'(x_n)\bigr|
\\
&& \label{131112}\qquad \quad{}+\bigl|Z^{2}(y_n)-Z^{2}(y)\bigr|+
\bigl|Z^{2}(x_n)-Z^{2}(x)\bigr|
\\
\nonumber
&&\qquad\quad{}+|y_n-y|\times \bigl|V'(x)\bigr|+|x_n-x|
\times \bigl|V'(x)\bigr|
\\
&&\qquad\quad{}+|y_n-x_n|\times \bigl|V'(x)-V'(x_n)\bigr|.\nonumber
\end{eqnarray}
In what follows fix
%
\begin{equation}
\label{13112} n=k_0N^*.
\end{equation}

Then we have
%
\begin{eqnarray}
\nonumber
&& |y_n-y|\times \bigl|V'(x)\bigr|+|x_n-x|
\times \bigl|V'(x)\bigr|
\\
\nonumber
&& \qquad \leq 2\cdot 2^{-k_0N^*}\bigl|V'(x)\bigr|
\\
&& \qquad \leq 2\cdot 2^{-N^*(\eta-2\alpha\gamma)}2^{-N^*}\bigl|V'(x)\bigr|
\qquad (\mbox{since } \eta\leq 2, k_0\geq 3)
\nonumber
\\[-8pt]
\label{13117}
\\[-8pt]
\nonumber
&& \qquad \leq \bigl(2|x-y|\bigr)^{\eta-2\alpha\gamma}2^{1-N^*}\bigl|V'(x)\bigr|
\qquad [\mbox{by }(\ref{13111})]
\\
\nonumber
&&\qquad\leq \bigl(|x-y|\bigr)^{\eta-2\alpha\gamma}2^{2-N^*}\bigl|V'(x)\bigr|
\\
&&\qquad\leq|x-y|^{\eta-2\alpha\gamma}, \qquad\forall n\geq N'_1,\nonumber
\end{eqnarray}
where the last inequality follows from $N^*\geq \log_2(|V'(x)|)+2$.
Now by triangle inequality and~(\ref{13112}) we get
%
\begin{eqnarray}
\nonumber
|x_n-y_n| &\leq & 2^{1-n}+|x-y|
\\
\label{13113}&\leq & 2^{1-N^*}
\\
\nonumber
&\leq & 4|x-y|.
\end{eqnarray}
This, (\ref{13112})  and~(\ref{eq14104}) imply
%
\begin{eqnarray}
\nonumber
|y_n-x_n|\cdot \bigl|V'(x)-V'(x_n)\bigr|
&\leq & 4|x-y|\cdot C_{(\fontsize{8.36}{10.36}{\selectfont\ref{11031}})}(J) 2^{-N^*k_0(\eta-1-2\alpha\gamma)}
\\
\label{13119} &\leq & 8C(J) |x-y|\cdot |x-y|^{\eta-1-2\alpha\gamma}
\\
\nonumber
&\leq & 8C(J) |x-y|^{\eta-2\alpha\gamma}.
\end{eqnarray}

Now recall that $Z^2$ is H\"older continuous with any  exponent less  than $\eta_{\mathrm{c}}$ (see Theorem~2 in \cite{FMW10})
to get that there exists $C=C(\omega)$ such that
\begin{eqnarray*}
&&\bigl|Z^{2}(y_n)-Z^{2}(y)\bigr|+ \bigl|Z^{2}(x_n)-Z^{2}(x)\bigr|
\\
&& \qquad \leq C(\omega) \bigl(|y_n-y|^{\eta_c-(2\alpha\gamma)/k_0}
+|x_n-x|^{\eta_c-(2\alpha\gamma)/k_0} \bigr)\qquad \forall y\in B \bigl(x,
2^{-N_3} \bigr).
\end{eqnarray*}
Recalling that
%
\begin{equation}
\label{13114} |y_n-y|, |x_n-x|\leq 2^{-n}
\end{equation}
 and~(\ref{13112}) we get
\begin{eqnarray*}
\bigl|Z^{2}(y_n)-Z^{2}(y)\bigr|+ \bigl|Z^{2}(x_n)-Z^{2}(x)\bigr|
&\leq & 2C(\omega)2^{-n(\eta_c-(2\alpha\gamma)/k_0)}
\\
&\leq & 2C(\omega)2^{-N^*(k_0\eta_c-2\alpha\gamma)}
\\
&\leq & 2C(\omega)2^{-N^*(\eta-2\alpha\gamma)},
\end{eqnarray*}
where the last inequality follows since by assumption $k_0>1+1/(\alpha-1-\beta)$ and hence $k_0\eta_c>\bar\eta_c>\eta$. By~(\ref{13111}), we immediately get
%
\begin{equation}
\label{131110} \bigl|Z^{2}(y_n)-Z^{2}(y)\bigr|+
\bigl|Z^{2}(x_n)-Z^{2}(x)\bigr| \leq 8C(
\omega)|x-y|^{\eta-2\alpha\gamma}.
\end{equation}
Now use again~(\ref{13114}) and triangle inequality to get that
\[
|y_n-x|\leq |y_n-y|+|y-x|\leq 2^{-(N^*-1)}.
\]
This together with the~(\ref{13114}) and the definition of $\tilde{S}_{\eta}(N,J)$ implies that
%
\begin{equation}
\label{13115} (x_n,y_n)\in \tilde{S}_{\eta}
\bigl(N^*-1,J \bigr)\cap \bigl\{ \bigl(i2^{-n}, j2^{-n}
\bigr), i,j\in \Z \bigr\}.
\end{equation}
Note that
%
\begin{eqnarray}
\nonumber
|x_n-y_n| &\leq & 2^{-(N^*-1)}\qquad [\mbox{by }(\ref{13113})]
\\
\label{13116} &\leq& 2^{-\log^2(k_0N^*)}
\\
\nonumber
& =& 2^{-\log^2(n)},
\end{eqnarray}
where the first inequality follows by~(\ref{13113}) and the second inequality
follows easily by our assumption $N^*\geq (k_0)^{10}, k_0>3$.
By (\ref{13115}), (\ref{13116}) and since  $N^*-1\geq N_1$, we can apply
Lemma~\ref{Step3}
to get
\begin{eqnarray}
\nonumber
&& \bigl|Z^{2}(y_n)-Z^{2}(x_n)-(y_n-x_n)V'(x_n)\bigr|
\\
\nonumber
&& \qquad =\bigl|Z^{2,\eta}_t(x_n,y_n)\bigr|
\nonumber
\\[-8pt]
\label{131111}
\\[-8pt]
\nonumber
&& \qquad \leq C|y_n-x_n|^{\eta-2\alpha\gamma}
\\
&& \qquad \leq C'|y-x|^{\eta-2\alpha\gamma},\nonumber
\end{eqnarray}
where the last inequality follows by~(\ref{13113}).

By~(\ref{131112}) and  the bounds (\ref{13117}), (\ref{13119}), (\ref{131110}), (\ref{131111}), we complete the proof.
\end{pf*}

\begin{lemma}
\label{Step4}
Let $J\geq 1$, $\eta\in (\eta_c,\min\{\bar\eta_{\mathrm{c}},1\})$.
For almost every $\omega\in A^\varepsilon$
and for
all $x\in (0,1)\setminus S_\eta(J)$,
\[
\bigl|Z^{2}(y)-Z^2(x)\bigr|\leq C_1(J)|y-x|^{\eta-2\alpha\gamma}
\qquad \forall y\in B\bigl(x, 2^{-N_1}\bigr),
\]
where $N_1=N_1(\omega)$ is from Lemma~\ref{Step3}.
\end{lemma}

\begin{pf}
For any $n\geq 1$,
take $x_n,y_n\in \{i2^{-n},i\in\Z\}$ satisfying the following conditions:
\begin{eqnarray*}
&& |x_n-x|\leq 2^{-n},\qquad |x_n-x_{n+1}|
\leq 2^{-n-1},
\\
&& |y_n-y|\leq 2^{-n},\qquad |y_n-y_{n+1}|
\leq 2^{-n-1}.
\end{eqnarray*}
Applying Lemma~\ref{Step3}, we get for every $N\geq N_1$ the bound
%
\begin{eqnarray}
\nonumber
&& \bigl|Z^{2,\eta}_t(y, x)\bigr|
\\
&& \qquad \leq \bigl| Z^{2,\eta}_t(y_N,
x_N)\bigr|+\sum_{n=N}^\infty \bigl(\bigl|
Z^{2,\eta}_t(y_{N+1}, y_{N})\bigr|+\bigl|Z^{2,\eta}_t(x_{N+1},
x_{N})\bigr| \bigr)
\nonumber
\\[-8pt]
\\[-8pt]
\nonumber
&& \qquad \leq 2C(J)2^{-N(\eta-\eta_c)} \Biggl(|x_N-y_N|^{\eta_c-2\alpha\gamma}+2
\sum_{n=N}^\infty 2^{-n(\eta_c-\alpha\gamma)} \Biggr)
\\
&& \qquad \leq C'(J)2^{-N(\eta-\eta_c)} \bigl(|x_N-y_N|^{\eta_c-2\alpha\gamma}+2^{-N(\eta_c-2\alpha\gamma)}
\bigr).
\nonumber
\end{eqnarray}
Choosing $N$ so that $|y-x|\in[2^{-N-1},2^{-N}]$, we complete the proof.
\end{pf}

Now we are able to complete the following.

\begin{pf*}{Proof of Lemma~\ref{P1}} Lemmas \ref{lemStep7}, \ref{Step4} imply that
\[
\mathbf{P}\bigl(H_{Z^2}(x)\geq \eta-2\alpha\gamma,  \forall x\in
(0,1)\setminus S_{\eta}(J);A^\varepsilon\bigr)\geq 1-\varepsilon.
\]
Letting here $\varepsilon\to0$, we complete the proof of the lemma.
\end{pf*}

\section{Lower bound for the Hausdorff dimension}
\label{sec5}

The aim of this section is to prove the following proposition.

\begin{proposition}
\label{prop08023}
For every $\eta\in(\eta_c,\overline{\eta}_c)\setminus\{1\}$,
\[
\operatorname{dim}({\mathcal E}_{Z^2,\eta})\geq (1+\beta) (\eta-
\eta_c),\qquad \mathbf{P}\mbox{-a.s. on }\bigl\{X_t
\bigl((0,1)\bigr)>0\bigr\}.
\]
\end{proposition}

\begin{remark}
Clearly, the above proposition together with Proposition~\ref{prop0802}
completes the proof of Theorem~\ref{thmmfractal}.
\end{remark}

As we have already mentioned in the \hyperref[sec1]{Introduction}, the proof of the lower
bound is much more involved then the proof of the upper one.  Due to the
mentioned complexity of the proof we give, for the reader's convenience, a
short description of our strategy. Section~\ref{sec4.1} is devoted to
deriving some uniform estimates on ``masses'' of~$X_s$ of dyadic intervals
at times $s$ close to $t$. In Section~\ref{sec4.2}, we construct a set
$\tilde{J}_{\eta,1}$ with
$\operatorname{dim}(\tilde{J}_{\eta,1})\geq (\beta+1)(\eta-\eta_{\mathrm{c}})$, on
which we show existence of ``big'' jumps of~$X$ that occur close to time $t$.
These jumps are ``encoded'' in the jumps of the auxiliary processes $L^+_{n,l,r}$
and they, in fact,  ``may'' destroy  the H\"older continuity of $X_t(\cdot)$
on $\tilde{J}_{\eta,1}$ for any  index greater or equal to $\eta$ (see
Lemma~\ref{lem321} and Lemma~\ref{lower4}). However, there are also other
jumps of the process $X$ (they will be encoded in processes $L^-_{n,l,r}$) which
may compensate the impact of the jumps of $X$ encoded in $L^+$. The most
difficult part of the proof is to show that there is no such compensation,
and this is done in Section~\ref{sec4.3}.  More precisely, we prove in
Section~\ref{sec4.3} that such a compensation is possible on a set of
the Hausdorff dimension strictly smaller than $(\beta+1)(\eta-\eta_{\mathrm{c}})$,
and hence does not influence the dimension result. It is done in Lemmata
\ref{compens}, \ref{lastlemma} and \ref{lem0271}.

\subsection{Uniform estimates for values of $X_s$ on dyadic intervals}
\label{sec4.1}

In this subsection, we derive some bounds for $X_s(I_k^{(n)})$, where
\[
I_k^{(n)}:= \bigl[ k2^{-n}, (k+1)2^{-n}\bigr).
\]

In what follows, fix some
%
\begin{equation}
\label{eq02081} m>3/\alpha,
\end{equation}
and let $\theta\in (0,1)$ be arbitrarily small.
Define
\begin{eqnarray*}
O_n&:=& \Bigl\{\omega\dvtx \mbox{there exists }k\in \bigl[0,
2^{n}-1 \bigr]\mbox{ such that }
\\
&& \hspace*{7pt}\sup_{s\in(t-2^{-\alpha n}n^{\alpha^2m/3},t)}X_s \bigl(I^{(n)}_k
\bigr)\geq 2^{-n}n^{2m\alpha/3} \Bigr\}
\end{eqnarray*}
and
\begin{eqnarray*}
B_n=B_n(\theta)& :=& \Bigl\{\omega\dvtx \mbox{there
exists }k \in \bigl[0, 2^{n}-1 \bigr]\mbox{ such that }
\\
&& \hspace*{6pt} I^{(n)}_k\cap \bigl\{x\dvtx X_t(x)\geq
\theta \bigr\}\neq\varnothing
\\
&& \hspace*{7pt} \mbox{and }\inf_{s\in(t-2^{-\alpha n}n^{-\alpha m},t)}X_s \bigl(I^{(n)}_k
\bigr)\leq 2^{-n}n^{-2m} \Bigr\}.
\end{eqnarray*}

\begin{lemma}\label{lower0}
There exists a constant $C$ such that
\[
\mathbf{P}(O_n)\leq Cn^{-m\alpha/3},\qquad n\geq 1.
\]
\end{lemma}

The proof is an almost word-by-word repetition of the
proof of Lemma~5.5 in \cite{FMW10}, and we omit it.

\begin{lemma}\label{lower1}
There exists a constant $C=C(m)$ such that, for every $\theta\in(0,1)$,
\[
\mathbf{P}\bigl(B_n(\theta)\cap A^{\varepsilon}\bigr)\leq C
\theta^{-1}n^{-\alpha m/3}, \qquad n\geq \tilde n(\theta),
\]
for some $\tilde n(\theta)$ sufficiently  large.
\end{lemma}

\begin{pf}
Define
\begin{eqnarray*}
\tau_n &:=& \inf \bigl\{s\in \bigl(t-2^{-\alpha n}n^{-\alpha m},t
\bigr)\dvtx X_s \bigl(I^{(n)}_k \bigr)\leq
2^{-n}n^{-2m}
\\
&& \hspace*{83pt}\qquad \mbox{for some }k\in \bigl[0,2^n-1 \bigr] \bigr\}.
\end{eqnarray*}
Fix an arbitrary $\theta\in(0,1)$.
If $\omega\in B_n=B_n(\theta)$, then there exists a sequence $\{(s_j,I^{(n)}_{k_j})\}$
such that
$X_{s_j}(I^{(n)}_{k_j})\leq 2^{-n}n^{-2m}$ for all $j\geq1$, and $s_j\downarrow\tau_n$, as $j\rightarrow\infty$.
Since for each
$n\geq1$ the number of intervals $I^{(n)}_k$ is finite, there exist $\tilde k_n$ and a
subsequence $j_r$ such that $k_{j_r}=\tilde k_n$ for all $r\geq1$. Therefore,
\[
\lim_{r\rightarrow\infty} X_{s_{j_r}}\bigl(I^{(n)}_{\tilde k_n}
\bigr)\leq 2^{-n}n^{-2m}.
\]
By the right continuity of the measure valued process $\{X_t\}_{t\geq 0}$,
we get that
\[
X_{\tau_n}\bigl(I^{(n)}_{\tilde{k}_n}\setminus \bigl\{
\tilde{k}_n 2^{-n}\bigr\}\bigr)\leq 2^{-n}n^{-2m}.
\]
Since $X$ has only positive jumps in the form of atomic measures and these
jumps do not occur with probability one at dyadic rational points of space,
we immediately deduce that,  in fact,
\[
X_{\tau_n}\bigl(I^{(n)}_{\tilde k_n}\bigr)\leq
2^{-n}n^{-2m}, \qquad \mathbf{P}\mbox{-a.s.}
\]
Put
\[
\tilde{B}_n:= \biggl[\frac{\tilde{k}_n}{2^{n}}+\frac{1}{2^{n+1}}-2^{-n}n^{-m},
\frac{\tilde{k}_n}{2^{n}}+\frac{1}{2^{n+1}}+2^{-n}n^{-m} \biggr],
\]
and
\[
\tilde E^{(n)}:= \bigl\{\omega\dvtx I^{(n)}_{\tilde{k}_n}
\cap\bigl\{x \dvtx X_t(x)>\theta\bigr\} \neq\varnothing \bigr\}.
\]
Recall that, on $A^\varepsilon$, $X_t(\cdot)$ is locally H\"older continuous on $(0,1)$
with\vspace*{1pt} exponent $\eta_c-\varepsilon$ and H\"older constant $C_{\fontsize{8.36}{10.36}{\selectfont\eqref{eq080220}}}$
[see \eqref{goodset}]. Therefore, on the event $A^{\varepsilon}\cap \tilde E^{(n)}$,
we have $X_t(x)\geq \theta/2$ for all $x\in\tilde{B}_n$ and all
$n\geq \tilde n(\theta)$, where $\tilde n(\theta)$ is chosen sufficiently large.

Thus, for $n\geq \tilde n(\theta)$,
\begin{eqnarray}
\nonumber
\theta 2^{-n}n^{-m}\mathbf{P} \bigl(
\tau_n<t,O_n^{c}, \tilde
E^{(n)}, A^{\varepsilon} \bigr) &=& \frac{\theta}{2}|
\tilde{B}_n|\mathbf{P} \bigl(\tau_n<t,O_n^{c},
\tilde E^{(n)},A^{\varepsilon} \bigr)
\\
\label{eq15101}&\leq& \mathbf{E} \bigl[X_t(
\tilde{B}_n){ \mathsf 1}_{\{\tau_n<t,O_n^{c},\tilde E^{(n)},A^{\varepsilon}\}} \bigr]
\\
&\leq& 
\mathbf{E} \bigl[X_t( \tilde{B}_n){
\mathsf 1}_{\{\tau_n<t,\tilde O_n^{c}\}} \bigr],\nonumber
\end{eqnarray}
where
\[
\tilde O_n^c :=\bigl\{ X_{\tau_n}
\bigl(I^{(n)}_k\bigr)\leq 2^{-n}n^{2m\alpha/3},
\mbox{for all } k=0,\ldots, 2^n-1\bigr\},
\]
and the last inequality in (\ref{eq15101}) follows since $ O_n^c\subset  \tilde O_n^c$.

Using the strong Markov property, we then obtain
%
\begin{equation}
\label{lower1.1}\hspace*{6pt} \theta 2^{-n}n^{-m}\mathbf{P} \bigl(
\tau_n<t,O_n^{c},\tilde
E^{(n)},A^{\varepsilon} \bigr)\leq \mathbf{E} \bigl[S_{t-\tau_n}X_{\tau_n}(
\tilde{B}_n){\mathsf 1}_{\{\tau_n<t,\tilde O_n^{c}\}} \bigr]e^{|a|t},
\end{equation}
for all $n\geq \tilde n(\theta)$. 
It is clear that
%
\begin{eqnarray}
\label{eq02082} && \mathbf{E} \bigl[S_{t-\tau_n}X_{\tau_n}(
\tilde{B}_n){\mathsf 1}_{\{\tau_n<t,\tilde O_n^{c}\}} \bigr]
\\
&& \qquad =\mathbf{E} \biggl[\int_{\mathsf{R}}X_{\tau_n}(dz)
\int_{\tilde{B}_n}p_{t-\tau_n}^\alpha(y-z)\,dy{\mathsf
1}_{\{\tau_n<t,\tilde O_n^{c}\}} \biggr] \qquad \forall n\geq 1.
\end{eqnarray}
Since $X_{\tau_n}(I^{(n)}_{\tilde{k}_n})\leq 2^{-n}n^{-2m}$ on the event $\{\tau_n<t\}$, we have
%
\begin{equation}
\label{lower1.2}\hspace*{14pt}\quad \mathbf{E} \biggl[\int_{I^{(n)}_{\tilde{k}_n}}X_{\tau_n}(dz)
\int_{\tilde{B}_n}p_{t-\tau_n}^\alpha(y-z)\,dy{\mathsf
1}_{\{\tau_n<t,\tilde O_n^{c}\}} \biggr] \leq 2^{-n}n^{-2m} \qquad \forall
n\geq 1.
\end{equation}
Recalling that $\tau_n\geq t-2^{-\alpha n}n^{-\alpha m}$ and using the scaling property of the kernel $p^\alpha$
together with the bound~(\ref{eq326})
we get
\begin{eqnarray*}
p^\alpha_{t-\tau_n}(y-z)&=& (t-\tau_n)^{-1/\alpha}p^\alpha_1
\biggl(\frac{y-z}{(t-\tau_n)^{1/\alpha}} \biggr)
\\
&\leq & C(t-\tau_n)|y-z|^{-\alpha-1}
\\
&\leq & C2^{-\alpha n}n^{-\alpha m}|y-z|^{-\alpha-1}.
\end{eqnarray*}
Further, if $z\in I^{(n)}_{\tilde{k}_n\pm j}$ and $y\in\tilde{B}_n$, then
\begin{eqnarray*}
|y-z| &\geq & (j-1)2^{-n}+ \bigl(1/2-n^{-m}
\bigr)2^{-n}= \bigl(j-1/2-n^{-m} \bigr)2^{-n}
\\
&\geq& \tfrac{1}{10}j2^{-n}\qquad  \forall n\geq 2, j\geq 1.
\end{eqnarray*}
Combining the last two bounds, we get
\begin{eqnarray*}
&& \int_{I^{(n)}_{\tilde{k}_n\pm j}}X_{\tau_n}(dz)\int_{\tilde{B}_n}p_{t-\tau_n}^\alpha(y-z)
\,dy
\\
&& \qquad \leq \int_{I^{(n)}_{\tilde{k}_n\pm j}}X_{\tau_n}(dz) \int
_{\tilde{B}_n} Cj^{-\alpha-1}2^{(\alpha+1)n}2^{-\alpha n}n^{-\alpha m}
\,dy
\\
&& \qquad =Cj^{-\alpha-1}n^{-(\alpha+1)m}X_{\tau_n}
\bigl(I^{(n)}_{\tilde{k}_n\pm j} \bigr)\qquad \forall n\geq 2, j\geq 1.
\end{eqnarray*}
On the event $\{\tau_n<t\}\cap \tilde O_n^{c}$ we then have
%
\begin{eqnarray}
\nonumber
&& \int_{I^{(n)}_{\tilde{k}_n-j}\cup I^{(n)}_{\tilde{k}_n+j}}X_{\tau_n}(dz)\int
_{\tilde{B}_n}p_{t-\tau_n}^\alpha(y-z)\,dy
\\
&& \label{eq02083}\qquad \leq Cj^{-\alpha-1}2^{-n} n^{-(\alpha+1)m+2m\alpha/3}
\\
&& \qquad =Cj^{-\alpha-1}2^{-n} n^{-(({1}/{3})\alpha+1)m} \qquad\forall n
\geq 2, j\geq 1.
\nonumber
\end{eqnarray}
Consequently, by summing up~(\ref{eq02083}) over $j\geq 1$, we get
%
\begin{eqnarray}
\quad&&\mathbf{E} \biggl[\int_{\mathsf{R}\setminus I^{(n)}_{\tilde{k}_n}}
X_{\tau_n}(dz)\int_{\tilde{B}_n}p_{t-\tau_n}^\alpha(y-z)
\,dy{ \mathsf 1}_{\{\tau_n<t,\tilde O_n^{c}\}} \biggr] \leq C2^{-n} n^{-(({1}/{3})\alpha+1)m},
\nonumber
\\[-8pt]
\label{lower1.3}\\[-8pt]
\eqntext{n\geq 2.}
\end{eqnarray}
This and (\ref{lower1.2}) imply that (\ref{eq02082}) is bounded by
%
\begin{equation}
\label{eq02085} C2^{-n} \bigl(n^{-(({1}/{3})\alpha+1)m}+n^{-2m}
\bigr), \qquad n\geq 2.
\end{equation}

Combining (\ref{lower1.1}), (\ref{eq02085}) and using the trivial bound for $n=1$, we obtain
\[
\theta\mathbf{P}\bigl(\tau_n<t,O_n^{c},
\tilde E^{(n)},A^{\varepsilon}\bigr) \leq C\bigl(n^{-({1}/{3})\alpha m}+n^{-m}
\bigr)\leq Cn^{-({1}/{3})\alpha m}, \qquad n\geq \tilde n(\theta). 
\]
In view of Lemma~\ref{lower0},
\begin{eqnarray*}
\mathbf{P} \bigl(\tau_n<t, \tilde E^{(n)},
A^{\varepsilon} \bigr) &\leq & \mathbf{P}(O_n)+\mathbf{P} \bigl(
\tau_n<t,O_n^{c}, \tilde
E^{(n)},A^{\varepsilon} \bigr)
\\
&\leq & C\theta^{-1}n^{-\alpha m/3}, \qquad n\geq \tilde n(\theta).
\end{eqnarray*}
This completes the proof of the lemma.
\end{pf}

\subsection{Analysis of the set of jumps which destroy the H\"older continuity}
\label{sec4.2}

In this subsection, we construct a set $\tilde{J}_{\eta,1}$ such that its
Hausdorff dimension is bounded from below by  $(\beta+1)(\eta-\eta_{\mathrm{c}})$
and in the vicinity of each $x\in\tilde{J}_{\eta,1}$ there are jumps of $X$
which destroy the H\"older continuity at $x$ for any index greater than $\eta$.

We first introduce $\tilde{J}_{\eta,1}$ and prove the lower bound for its
dimension.
Set
\[
q:=\frac{(\alpha+3)m}{(\beta+1)(\eta-\eta_c)}
\]
and define
\begin{eqnarray*}
A_{k}^{(n)} &:=& \bigl\{\bDelta X_s
\bigl(I^{(n)}_{k-2n^q-2} \bigr)\geq 2^{-(\eta+1)n}
\\
&&\hspace*{6pt}\mbox{for some } s\in\bigl[t-2^{-\alpha n}n^{-\alpha m},t-2^{-\alpha(n+1)}(n+1)^{-\alpha m}\bigr)
\bigr\},
\\
J^{(n)}_{k,r} &:=& \biggl[\frac{k}{2^n}-
\bigl(n^q2^{-n}\bigr)^r,\frac{k+1}{2^n}+
\bigl(n^q2^{-n}\bigr)^r \biggr].
\end{eqnarray*}

Let us introduce the following notation.
For a Borel set $B$ and an event $E$, define a random set
\[
B\mathsf{1}_E(\omega):= \cases{ B, & \quad $\omega\in E$,
\cr
\varnothing, & \quad $\omega\notin E$.}
\]
Now we are ready to define random sets
\[
\tilde{J}_{\eta,r}:=\limsup_{n\to\infty} \bigcup
_{k=2n^q+2}^{2^n-1} J^{(n)}_{k,r}{
\mathsf 1}_{A_k^{(n)}}, \qquad r>0.
\]
As we have mentioned already, we are interested in getting the lower bound
on Hausdorff dimension of $\tilde{J}_{\eta,1}$. The standard procedure
for this is as follows. First,  show that a bit ``inflated'' set
$\tilde{J}_{\eta,r}$, for certain $r\in (0,1)$, contains open intervals.
This would imply a lower bound $r$ on the Hausdorff dimension of
$\tilde{J}_{\eta,1}$ (see  Lemma~\ref{lower2} and Theorem~2 from
\cite{Jaff99} where a similar\vspace*{1pt} strategy was implemented). Thus,  to get a
sharper bound on Hausdorff dimension of $\tilde{J}_{\eta,1}$, one should
try to take $r$ as large as possible. In the next lemma, we show that,
in fact, one can choose $r=(\beta+1)(\eta-\eta_{\mathrm{c}})$.

\begin{lemma}
\label{lower2}
On the event $A^{\varepsilon}$,
\[
\bigl\{x\in (0,1)\dvtx X_t(x)\geq\theta\bigr\}\subseteq
\tilde{J}_{\eta,(\beta+1)(\eta-\eta_{\mathrm{c}})}, \qquad \pr\mbox{-a.s.}
\]
for every $\theta\in(0,1)$.
\end{lemma}

\begin{pf}
Fix an arbitrary $\theta\in(0,1)$.
We estimate the probability of the event $E_n\cap A^{\varepsilon}$, where
\[
E_n:= \Biggl\{\omega \dvtx \bigl\{x\in (0,1)\dvtx X_t(x)
\geq\theta\bigr\}\subseteq \bigcup_{k=2n^q+2}^{2^n-1}
J^{(n)}_{k,(\beta+1)(\eta-\eta_c)}{\mathsf 1}_{A_k^{(n)}} \Biggr\}.
\]
It follows from Lemma~\ref{lower1} that, for all
$n\geq \tilde n(\theta)$,
\begin{eqnarray}
\mathbf{P} \bigl(E_n^c\cap A^{\varepsilon} \bigr)&
\leq& \mathbf{P} \bigl(E_n^c\cap B_n\cap
A^{\varepsilon} \bigr)+ \mathbf{P} \bigl(E_n^c \cap
B_n^{c}\cap A^{\varepsilon} \bigr)
\nonumber
\\[-8pt]
\label{lower3.1}
\\[-8pt]
\nonumber
&\leq & C\theta^{-1}n^{-\alpha m/3}+\mathbf{P}
\bigl(E_n^c \cap B_n^{c}\cap
A^{\varepsilon} \bigr).
\end{eqnarray}
For any $k=0,\ldots,2^n-1$, the compensator measure $\widehat{N}(dr,dy,ds)$
of the random measure $\mathcal{N}(dr,dy,ds)$
[the jump measure for $X$---see discussion after~(\ref{incr})],
on
\begin{eqnarray*}
&& \mathcal{J}^{(n)}_1\times
I^{(n)}_k\times \mathcal{J}^{(n)}_2
\\
&& \qquad := \bigl[2^{-(\eta+1)n},\infty\bigr)\times I^{(n)}_k
\times \bigl[t-2^{-\alpha n}n^{-\alpha m},t-2^{-\alpha(n+1)}(n+1)^{-\alpha m}\bigr),
\end{eqnarray*}
is given by the formula
%
\begin{equation}
\label{eq302} 1 \bigl\{(r,y,s)\in \mathcal{J}^{(n)}_1
\times I^{(n)}_k\times \mathcal{J}^{(n)}_2
\bigr\}\varrho r^{-2-\beta}\,dr X_s(dy) \,ds.
\end{equation}
If
\[
k\in K_{\theta}:= \bigl\{l\dvtx I^{(n)}_l\cap
\bigl\{x\in (0,1) \dvtx X_t(x)\geq\theta\bigr\}\neq\varnothing\bigr\},
\]
then, 
by the definition of $B_n$, we have
\begin{equation}
\label{eq301} X_s \bigl(I^{(n)}_k \bigr)
\geq 2^{-n}n^{-2m}, \qquad \mbox{for $s\in
\mathcal{J}^{(n)}_2$, on the event $A^{\varepsilon}\cap
B_n^{c}$.}
\end{equation}
Define the measure $\widehat{\Gamma}(dr,dy,ds)$ on  $\R_+\times (0,1)\times \R_+ $, as follows:
%
\begin{equation}
\label{eq02086} \widehat{\Gamma}(dr,dy,ds):= \varrho r^{-2-\beta}\,dr
n^{-2m} \,dy \,ds.
\end{equation}
Then, by~(\ref{eq302}) and~(\ref{eq301}), on $A^{\varepsilon}\cap B_n^{c}$, and on the set
\[
\mathcal{J}^{(n)}_1\times \bigl\{y\in (0,1) \dvtx
X_t(y)\geq\theta\bigr\}\times \mathcal{J}^{(n)}_2
\]
we have the following  bound:
\[
\widehat{\Gamma}\bigl(dr,I_k^{(n)},\mathcal{J}^{(n)}_2
\bigr)\leq \widehat{\mathcal{N}}\bigl(dr,I_k^{(n)},
\mathcal{J}^{(n)}_2\bigr),\qquad k\in K_\theta.
\]
By standard arguments, it is easy to construct the Poisson point process $\Gamma(dr,dx, ds)$
on $\R_+\times (0,1)\times \R_+$ with intensity measure $\widehat\Gamma$ given by~(\ref{eq02086}) on the whole space
$\R_+\times (0,1)\times \R_+$ such that on $A^{\varepsilon}\cap B_n^{c}$,
\[
\Gamma \bigl(dr,I_k^{(n)},\mathcal{J}^{(n)}_2
\bigr)\leq\mathcal{N} \bigl(dr,I_k^{(n)},
\mathcal{J}^{(n)}_2 \bigr)
\]
for $r\in \mathcal{J}^{(n)}_1$ and $k\in K_\theta$.

Now, define
\[
\xi_k^{(n)}= {\mathsf 1}_{ \{\Gamma (\mathcal{J}^{(n)}_1\times I^{(n)}_{k-2n^q-2}\times \mathcal{J}^{(n)}_2
 )\geq 1 \}},\qquad k\geq
2n^q+2.
\]
Clearly, on $A^{\varepsilon}\cap B_n^{c}$ and for $k$ such that $k-2n^q-2\in K_{\theta}$,
\[
\xi_{k}^{(n)}\leq {\mathsf 1}_{ A_k^{(n)}}.
\]
Moreover, by construction $\{\xi_k^{(n)}\}_{k=2n^q+2}^{2^n+2n^q+1}$ is a collection of independent
 identically distributed Bernoulli random variables with success probabilities
\begin{eqnarray*}
p^{(n)}&:=& \widehat\Gamma \bigl(\mathcal{J}^{(n)}_1
\times I^{(n)}_{k-2n^q-2}\times \mathcal{J}^{(n)}_2
\bigr)
\\
&\,=& C 2^{(\eta-\eta_c)(1+\beta)n-n}n^{-(\alpha+2)m}.
\end{eqnarray*}
From the above  coupling with the Poisson point process $\Gamma$, it is
easy to see that
%
\begin{equation}
\label{eq02087} \mathbf{P} \bigl(E_n^c\cap
B_n^{c}\cap A^{\varepsilon} \bigr)\leq \mathbf{P} \bigl(
\tilde{E}_n^{c} \bigr),
\end{equation}
where
\[
\tilde{E}_n:= \Biggl\{(0,1)\subseteq\bigcup
_{k=2n^q+2}^{2^n+2n^q+1}J^{(n)}_k{\mathsf
1}_{\{\xi^{(n)}_k=1\}} \Biggr\}.
\]

Let $L{(n)}$ denote the length of the longest run of zeros
in the sequence $\{\xi_k^{(n)}\}_{k=2n^q+2}^{2^n+2n^q+1}$.
Clearly,
\[
\mathbf{P}\bigl(\tilde{E}_n^c\bigr)\leq\mathbf{P}
\bigl(L^{(n)}\geq 2^{n-(\beta+1)(\eta-\eta_c)n}n^{m(\alpha+3)}\bigr)
\]
and it is also obvious that
\[
\mathbf{P}\bigl(L^{(n)}\geq j\bigr)\leq 2^{n}p^{(n)}
\bigl(1-p^{(n)}\bigr)^j\qquad \forall j\geq 1.
\]
Use this with the fact that,
by~(\ref{eq02081}),
$m>1$,
to get that
%
\begin{equation}
\label{eq02088} \mathbf{P} \bigl(\tilde{E}_n^c \bigr)
\leq \exp \bigl\{-\tfrac{1}{2}n^{m} \bigr\}
\end{equation}
for all $n$ sufficiently large. Combining \eqref{lower3.1}, \eqref{eq02087}
and \eqref{eq02088}, we conclude that the sequence
$\mathbf{P}(E_n^c\cap A^{\varepsilon})$ is summable.
Applying Borel--Cantelli, we complete the proof of the lemma.
\end{pf}

Define
\[
h_\eta(x):=x^{(\beta+1)(\eta-\eta_c)}\log^2\frac{1}{x}
\]
and
\[
\mathcal{H}_\eta(A):=\lim_{\epsilon\to0} \inf \Biggl\{\sum
_{j=1}^\infty h_\eta\bigl(|I_j|\bigr),
A\in\bigcup_{j=1}^\infty I_j
\mbox{ and }|I_j|\leq\epsilon \Biggr\}.
\]
Combining Lemma~\ref{lower2} and Theorem~2 from \cite{Jaff99}, one can easily
get.

\begin{corollary}
\label{lower3}
On the event $A^{\varepsilon}\cap\{X_t((0,1))>0\}$,
\[
\mathcal{H}_\eta(\tilde{J}_{\eta,1})>0, \qquad \pr\mbox{-a.s.}
\]
and, consequently, on $A^{\varepsilon}\cap\{X_t((0,1))>0\}$,
\[
\operatorname{dim}(\tilde{J}_{\eta,1})\geq (\beta+1) (\eta-
\eta_c),\qquad \pr\mbox{-a.s.}
\]
\end{corollary}

\begin{pf}
Fix any $\theta\in(0,1)$.
If $\omega\in A^\varepsilon$ is such that $B_\theta:=\{x\in(0,1) \dvtx  X_t(x)\geq\theta\}$
is not empty, then by the local H\"older continuity of $X_t(\cdot)$ there exists an
open interval $(x_1(\omega),x_2(\omega))\subset B_{\theta/2}$. Moreover, in view of
Lemma~\ref{lower2},
\[
\bigl(x_1(\omega),x_2(\omega)\bigr)\subset
\tilde{J}_{\eta,(\beta+1)(\eta-\eta_c)}(\omega),\qquad \pr\mbox{-a.s.}
\]
on the event $A^\varepsilon\cap\{B_\theta~\mathrm{is}~\mathrm{not}~\mathrm{empty}\}$. Thus, we may apply
Theorem~2 from \cite{Jaff99} to the set $(x_1(\omega),x_2(\omega))$, which gives
\[
\mathcal{H}_\eta\bigl(\bigl(x_1(\omega),x_2(
\omega)\bigr)\cap\tilde{J}_{\eta,1}\bigr)>0,\qquad \pr\mbox{-a.s.}
\]
on the event $A^\varepsilon\cap\{B_\theta \mbox{ is not empty}\}$.
Thus,
\[
\operatorname{dim}\bigl(\bigl(x_1(\omega),x_2(\omega)
\bigr)\cap\tilde{J}_{\eta,1}\bigr)\geq (\beta+1) (\eta-
\eta_c), \qquad \pr\mbox{-a.s.}
\]
on the event $A^\varepsilon\cap\{B_\theta \mbox{ is not empty}\}$. Due to the monotonicity
of $\mathcal{H}_\eta(\cdot)$ and $\operatorname{dim}(\cdot)$, we conclude that
$\mathcal{H}_\eta(\tilde{J}_{\eta,1})>0$ and
$\operatorname{dim}(\tilde{J}_{\eta,1})\geq (\beta+1)(\eta-\eta_c)$, $\pr$-a.s.
on the event $A^\varepsilon\cap\{B_\theta \mbox{ is not empty}\}$. Noting that
${\mathsf 1}_{\{B_\theta~\mathrm{is}~\mathrm{not}~\mathrm{empty}\}}\uparrow{\mathsf 1}_{\{X_t(0,1)>0\}}$ as \mbox{$\theta\downarrow0$},
$\pr$-a.s., we complete the proof.
\end{pf}

Now we turn to the second part of the present subsection.
By construction of $\tilde{J}_{\eta,1}$, we know that  to the
left of every point $x\in\tilde{J}_{\eta,1}$ there
exist big jumps of $X$ at time $s$ ``close'' to $t$: such
jumps are  defined by  the events $A^{(n)}_k$. We would like to show
that these jumps  will result in destroying the H\"older continuity of any index greater than $\eta$
at the point $x$. To this end, we will introduce auxiliary processes $L^{\pm}_{n,l,r}$  that are indexed by a
grid \textit{finer} than $\{k2^{-n}, k=0,1,\ldots\}$. That is, take some integer $Q>1$ (note, that eventually
$Q$ will be chosen large enough, depending on $\eta$). According to Lemma~2.15 from \cite{FMW10} [see also~(\ref{eq17101}), (\ref{eq17102})],
there exist spectrally positive  $(1+\beta)$-stable
processes $L^\pm_{n,l,r}$ such that
%
\begin{eqnarray}
&& \wZ^{2,\eta}_s \bigl(l2^{-Qn},r2^{-Qn}
\bigr)= L^+_{n,l,r} \bigl(T^{n,l,r}_+(s) \bigr)-
L^-_{n,l,r} \bigl(T^{n,l,r}_{-}(s) \bigr),
\nonumber
\\[-8pt]
\label{eq18102}
\\[-8pt]
\eqntext{0\leq l<r\leq 2^{Qn}, 0\leq s \leq t,}
\end{eqnarray}
where
\[
T^{n,l,r}_\pm(s)=\int_0^s
du \int_{\mathbb{R}}X_u(dy) \bigl( \bigl(
\tilde{p}^{\alpha,\eta}_{t-u} \bigl(l2^{-Qn}-y,r2^{-Qn}-y
\bigr) \bigr)^{\pm} \bigr)^{1+\beta}, \qquad s\leq t.
\]
The goal\vspace*{1pt} of the remaining part of this
subsection is to show that, in fact, ``big'' jumps  of $X$ defined via $A^{(n)}_k$ imply
``big'' values of $L^{+}_{n,l,r}$ for certain $l,r$.

We need to introduce additional  notation related to the event $A^{(n)}_k$. If $A^{(n)}_k$ occurs, then
there is a jump of the process $X$  at time $s^n_k$ such that
%
\begin{equation}
\label{eq02121a} \bDelta X_{s^n_k} \bigl(I_{k-2n^q-2}^{(n)}
\bigr) \geq 2^{-(\eta+1)n},
\end{equation}
and
%
\begin{equation}
\label{eq02121} s^n_k\in\bigl[t-2^{-\alpha n}n^{-\alpha m},
t-2^{-\alpha(n+1)}(n+1)^{-\alpha m}\bigr).
\end{equation}
Let $y^n_k\in I^{(n)}_{k-2n^q-2}$ denote the spatial position of that jump. Now put
\[
l^n_k= \bigl\lfloor 2^{Qn}y^n_k
\bigr\rfloor, 
\]
and for every $x\in (0,1)$ define
\[
\tilde{k}_n(x)=\bigl\lfloor 2^{Qn}x\bigr\rfloor.
\]
To simplify notation, in what follows, for any $n,l,r$, we denote by $\bDelta L^+_{n,l, r}$ the maximal jump of $L^+_{n,l, r}$, that is,
\[
\bDelta L^+_{n,l, r}:= \sup_{s\leq t} \bDelta
L^+_{n,l,r}\bigl(T^{n,l,r}_+(s)\bigr).
\]
Also set
%
\begin{equation}
\label{171003} \bL^\pm_{n,l, r} := L^\pm_{n,l,r}
\bigl(T^{n,l,r}_{\pm}(t) \bigr).
\end{equation}

Now we explain briefly why we look at fine dyadic intervals
$(l2^{-Qn},r2^{-Qn})$. First, we can not work directly with
increments of $Z^2$  at random points $x\in\tilde{J}_{\eta,1}$.
However, if we show that H\"older continuity is destroyed
at $2^{-Qn}\tilde{k}_n(x)$, then we will be able to infer
that, at the point $x$, the H\"older continuity of any
index greater than $\eta$  is destroyed as well: to show this
we need  $Q$ to be sufficiently large. Also, on this
finer scale, we can show that $L^+_{n,l^n_k, r}$ has ``large''\vspace*{-1.5pt} jumps for
all $r$ which are close to $\tilde{k}_n(x)$. This property is quite
important because of a possible compensation effect, which will be
investigated in the next subsection.

In the next two lemmas, we start to fulfill the above program. In Lemma~\ref{lem321}, we show that on the event
$ A^{(n)}_k$,  $L^+_{n,l^n_k, r}$ has ``large'' jumps of order $2^{-\eta n}n^m$ for all $r$'s sufficiently close to $\tilde{k}_n(x)$ with
$x\in J^{(n)}_{k,1}$.
As a consequence, in Lemma~\ref{lower4}, we obtain, by pretty standard arguments, that $\bL^+_{n,l^{n}_{k}, r}$ can
also take ``big'' values of the same order, for certain  $n,l^{n}_{k}, r$.
%

\begin{lemma}
\label{lem321}
Let $\eta\in(\eta_c,\bar\eta_c)$, and fix an arbitrary integer $R>0$.
There exist constants $C_{(\fontsize{8.36}{10.36}{\selectfont\ref{eq420}})}$ and
$N_{\fontsize{8.36}{10.36}{\selectfont\ref{lem321}}}$ sufficiently large, such that for all
$n\geq N_{\fontsize{8.36}{10.36}{\selectfont\ref{lem321}}}$ and all
$r^n_k\in \{\tilde{k}_{n}(x)-R, \tilde{k}_{n}(x)-R+1,\ldots,\tilde{k}_{n}(x)\}$,
\begin{equation}\label{eq420}
\qquad A^{(n)}_k \subset \bigl\{\bDelta L^+_{n,l^n_k, r^n_k}\geq
C_{(\fontsize{8.36}{10.36}{\selectfont\ref{eq420}})} 2^{-\eta n} n^m, \forall x\in
J^{(n)}_{k,1} \bigr\},\qquad k=2n^q+2,\ldots,
2^n.\hspace*{-12pt}
\end{equation}
 \end{lemma}

\begin{pf}
Fix $n$ sufficiently large (to be  chosen later)  and $k\in \{2n^q+2,\ldots, 2^n\}$.
In what follows, we assume that ${A_{k}^{(n)}}$ occurs. Then we have to show that
%
\[
\bDelta L^+_{n,l^n_k, r^n_k}\geq C_{(\fontsize{8.36}{10.36}{\selectfont\ref{eq420}})} 2^{-\eta n}
n^m\qquad \forall x\in J^{(n)}_{k,1}.
\]
Fix an arbitrary $x\in J^{(n)}_{k,1}$. Recall that $(s^n_k,y^n_k)$ denotes a space-time
location of a jump of $X$ that appears in the definition of  ${A_{k}^{(n)}}$.
To simplify notation, to the end of the lemma,  we will  suppress the superindex $n$ in $l^n_k, r^n_k, y^n_k$.
If $A^{(n)}_k$ occurred, then
%
\begin{equation}
\label{eq328} \bDelta L^+_{n,l_k, r_k}\geq 2^{-(\eta+1)n} \bigl(
\tilde{p}^{\alpha,\eta}_{t-s_k} \bigl(l_k2^{-Qn}-y_k,r_k
2^{-Qn}-y_k \bigr) \bigr)_{+}.
\end{equation}
So to verify the lemma, we have to obtain a suitable strictly positive  lower bound for
$\tilde{p}^{\alpha,\eta}_{t-s_k}(l_k2^{-Qn}-y_k,r_k 2^{-Qn}-y_k)$.

 First, we will obtain a lower bound for $ p^{\alpha}_{t-s_k}(l_k 2^{-Qn}-y_k)$:
%
\begin{eqnarray}
\hspace*{7pt} p^{\alpha}_{t-s_k} \bigl(l_k 2^{-Qn}-y_k
\bigr)&=& (t-s_k)^{-1/\alpha} p_1
\bigl((t-s_k)^{-1/\alpha} \bigl(l_k
2^{-Qn}-y_k \bigr) \bigr)
\nonumber
\\[-8pt]
\label{eq325}
\\[-8pt]
\nonumber
&\geq & 2^{n} n^m p_1
\bigl((t-s_k)^{-1/\alpha} \bigl(l_k
2^{-Qn}-y_k \bigr) \bigr),
\end{eqnarray}
where the last inequality follows by (\ref{eq02121}). By definition of $l_k$, we get
%
\begin{equation}
\bigl|l_k 2^{-Qn}-y_k\bigr| \leq 2^{-Qn},
\end{equation}
and this with again  (\ref{eq02121}) and monotonicity of $p_1(\cdot)$ implies
%
\begin{equation}
\label{eq324} \hspace*{13pt}p_1 \bigl((t-s_k)^{-1/\alpha}
\bigl(l_k 2^{-Qn}-y_k \bigr) \bigr) \geq
p_1 \bigl(2^{-(Q-1)n+1}(n+1)^m \bigr)  \geq
p_1(1)
\end{equation}
for all $n$  sufficiently large. Let $N_{(\fontsize{8.36}{10.36}{\selectfont\ref{eq324}})}$ be sufficiently large such that (\ref{eq324})
 holds for all $n\geq N_{(\fontsize{8.36}{10.36}{\selectfont\ref{eq324}})}$.  Then~(\ref{eq325}), (\ref{eq324}) imply
%
\begin{eqnarray}
p^{\alpha}_{t-s_k} \bigl(l_k 2^{-Qn}-y_k
\bigr) &\geq & 2^{n} p_1(1) n^m
\nonumber
\\[-8pt]
\label{eq322}
\\[-8pt]
\nonumber
&=& C_{(\fontsize{8.36}{10.36}{\selectfont\ref{eq322}})}2^{n}n^m \qquad \forall
n\geq N_{(\fontsize{8.36}{10.36}{\selectfont\ref{eq324}})}.
\end{eqnarray}

Next, by definition of $A_{k}^{(n)},  V_{k}^{(n)}$,  we easily get, for all sufficiently large $n$,
%
\begin{equation}
-y_k+ r_k 2^{-Qn} \geq 2^{-n}
\bigl(2n^q+1 \bigr)- \bigl(2^{-n}n^q+R2^{-Qn}
\bigr)\geq 2^{-n-1}n^q.
\end{equation}
Use this and the bound on  $t-s_k$ to get
%
\begin{eqnarray}
\nonumber
p^{\alpha}_{t-s_k} \bigl(r_k
2^{-Qn}-y_k \bigr)&= & (t-s_k)^{-1/\alpha}
p_1 \bigl((t-s_k)^{-1/\alpha} \bigl(r_k
2^{-Qn}-y_k \bigr) \bigr)
\\
\nonumber
&\leq & C_{(\fontsize{8.36}{10.36}{\selectfont\ref{eq326}})}(t-s_k)^{-1/\alpha}
\bigl((t-s_k)^{-1/\alpha} \bigl(r_k
2^{-Qn}-y_k \bigr) \bigr)^{-\alpha-1}
\\
\label{eq2123} &=& C_{(\fontsize{8.36}{10.36}{\selectfont\ref{eq326}})} (t-s_k)
\bigl(r_k 2^{-Qn}-y_k \bigr)^{-\alpha-1}
\\
\nonumber
&\leq & C_{(\fontsize{8.36}{10.36}{\selectfont\ref{eq326}})} 2^{-\alpha n} n^{-\alpha m}2^{(n+1)(\alpha+1)}n^{-q(\alpha+1)}
\\
\nonumber
&= & C_{(\fontsize{8.36}{10.36}{\selectfont\ref{eq2123}})} 2^{n}n^{-\alpha m-q(\alpha+1)}.
\end{eqnarray}
Next, we will bound from above the quantity
\[
\bigl\llvert \bigl(l_k 2^{-Qn}-r_k
2^{-Qn}\bigr)p_{t-s_k}^{\alpha,\prime}\bigl(r_k
2^{-Qn}-y_k\bigr)\bigr\rrvert,
\]
where $p_t^{\alpha,\prime}(z):=
\frac{\partial p^{\alpha}(z)}{\partial z}$.
It is easy to check that
%
\begin{eqnarray}
\nonumber
\bigl\llvert p_{t-s_k}^{\alpha,\prime} \bigl(r_k
2^{-Qn}-y_k \bigr) \bigr\rrvert &\leq & C
(t-s_k)^{-2/\alpha} p_{1}^{\alpha}
\bigl((t-s_k)^{-1/\alpha} \bigl(r_k
2^{-Qn}-y_k \bigr)/2 \bigr)
\\
\label{22312} &=& C (t-s_k)^{-1/\alpha} p_{t-s_k}^{\alpha}
\bigl( \bigl(r_k 2^{-Qn}-y_k \bigr)/2
\bigr)
\\
\nonumber
&\leq & C_{(\fontsize{8.36}{10.36}{\selectfont\ref{22312}})} 2^{n} n^{m}
2^{n}n^{-\alpha m-q(\alpha+1)},
\end{eqnarray}
where the last inequality follows by~(\ref{eq2123}) and the bound on $t-s_k$.
Since
\[
\bigl|l_k 2^{-Qn}-r_k 2^{-Qn}\bigr|\leq 3
\cdot 2^{-n} n^q,
\]
this implies
%
\begin{equation}
\label{eq323}\hspace*{4pt}\quad \bigl\llvert \bigl(l_k 2^{-Qn}-r_k
2^{-Qn} \bigr)p_{t-s_k}^{\alpha,\prime} \bigl(r_k
2^{-Qn}-y_k \bigr) \bigr\rrvert \leq3C_{(\fontsize{8.36}{10.36}{\selectfont\ref{22312}})}
2^nn^{-m(\alpha-1)-q\alpha}.
\end{equation}
Then by definition of $\tilde{p}^{\alpha,\eta}_{t}$, (\ref{eq322}), (\ref{eq2123}),
(\ref{eq323}), we immediately get that there exists an $N_{(\fontsize{8.36}{10.36}{\selectfont\ref{eq327}})}
\geq N_{(\fontsize{8.36}{10.36}{\selectfont\ref{eq324}})}$, such that, for any
$\eta\in(\eta_c,\bar\eta_c)$,
%
\begin{eqnarray}
\nonumber
&& \tilde{p}^{\alpha,\eta}_{t-s_k} \bigl(l_k
2^{-Qn}-y_k,r_k 2^{-Qn}-y_k
\bigr)
\\
\label{eq327} && \qquad \geq C_{(\fontsize{8.36}{10.36}{\selectfont\ref{eq322}})}2^{n}n^m-C_{(\fontsize{8.36}{10.36}{\selectfont\ref{eq2123}})}
2^{n}n^{-m\alpha-q(\alpha+1)}- 3C_{(\fontsize{8.36}{10.36}{\selectfont\ref{22312}})} 2^{n}n^{-m(\alpha-1)-q\alpha}
\\
&&
\nonumber
\qquad \geq \tfrac{1}{2}C_{(\fontsize{8.36}{10.36}{\selectfont\ref{eq322}})}2^{n}n^m
\qquad \forall n\geq N_{(\fontsize{8.36}{10.36}{\selectfont\ref{eq327}})}.
\end{eqnarray}
Substitute the above lower bound into~(\ref{eq328}) and the result follows immediately.
\end{pf}

\begin{lemma}\label{lower4}
Let $\eta\in(\eta_c,\bar\eta_c)$, and fix an arbitrary integer $R>0$.
On $A^{\varepsilon}$,
for every $x\in \tilde{J}_{\eta,1}$ there exists a (random) sequence $\{(n_j,k_j)\}$, such that
\[
\bL^+_{n_j,l^{n_j}_{k_j}, r^{n_j}_{k_j}}\geq C 2^{-\eta n_j}n_j^{m}
\]
for all $r^{n_j}_{k_j}\in [\tilde{k}_{n_j}(x)-R,\tilde{k}_{n_j}(x)-R+1,\ldots,  \tilde{k}_{n_j}(x) ]$.
\end{lemma}

\begin{pf}
Recall~(\ref{eqtime1}) to get that on $A^\varepsilon$, and for any $l,r$,
%
\begin{equation}
\label{eq181012} T^{n,l,r}_{+}(t) \leq \hat{T}^{n,l,r}
\bigl(l2^{-Qn},r2^{-Qn} \bigr)=C_{(\fontsize{8.36}{10.36}{\selectfont\ref{eqtime1}})}
\bigl(|r-l|2^{-Qn} \bigr)^{\alpha-\beta-\varepsilon_{1}},
\end{equation}
and take $\epsilon_1<(\bar{\eta}_c-\eta)(\beta+1)/2$. This, Lemma~\ref{lem321} and
 Lemma~2.4 from \cite{FMW10}, imply that for all $n$ sufficiently large, and for any $k\in [2n^q+2, 2^n-1]$,
\begin{eqnarray*}
&& \mathbf{P} \bigl(A^\varepsilon\cap A^{(n)}_k \cap
\bigl\{\bL_{n,l^n_k,r^n_k}^+\leq C_{(\fontsize{8.36}{10.36}{\selectfont\ref{eq420}})} 2^{-\eta n-1}n^{m}
\bigr\} \bigr)
\\
&& \qquad \leq \mathbf{P} \bigl(A^\varepsilon\cap \bigl\{ \bDelta
L^+_{n,l^n_k,r^n_k}\geq C_{(\fontsize{8.36}{10.36}{\selectfont\ref{eq420}})} 2^{-\eta n} n^m
\bigr\} \cap \bigl\{\bL_{n,l^n_k,r^n_k}^+\leq C_{(\fontsize{8.36}{10.36}{\selectfont\ref{eq420}})}
2^{-\eta n-1}n^{m} \bigr\} \bigr)
\\
&& \qquad\leq \mathop{\sum_{l\dvtx 2^{-Qn}l\in I^{(n)}_k,}}_{ r \dvtx  2^{-Qn}r\in J^{(n)}_{k,1}} \mathbf{P}
\bigl(A^\varepsilon
\cap \bigl\{ \bDelta L^+_{n,l, r}\geq
C_{(\fontsize{8.36}{10.36}{\selectfont\ref{eq420}})} 2^{-\eta n} n^m \bigr\}
\\
&& \hspace*{76pt}\qquad \quad
{}\cap \bigl\{
\bL_{n,l,r}^+\leq C_{(\fontsize{8.36}{10.36}{\selectfont\ref{eq420}})} 2^{-\eta n-1}n^{m}
\bigr\} \bigr)
\\
&&\qquad \leq \mathop{\sum_{l \dvtx  2^{-Qn}l\in I^{(n)}_k,}}_{r \dvtx  2^{-Qn}r\in J^{(n)}_{k,1}} \mathbf{P} \Bigl(\inf
_{s\leq \hat{T}^{n,l,r}(l2^{-Qn},r2^{-Qn})} L_{n,l,r}^+(s)\leq - C_{(\fontsize{8.36}{10.36}{\selectfont\ref{eq420}})}
2^{-\eta n-1}n^{m} \Bigr)
\\
&& \qquad \leq 2^{(Q-1)n}5n^q\exp \bigl\{-c2^{(\bar{\eta}_c-\eta)(\beta+1)n/(2\beta)}
\bigr\}.\\[-26pt]
\end{eqnarray*}
Consequently,
\begin{eqnarray*}
&& \mathbf{P} \bigl(A^\varepsilon\cap A^{(n)}_k\cap
\bigl\{\bL_{n,l^n_k,r^n_k}^+\leq C_{(\fontsize{8.36}{10.36}{\selectfont\ref{eq420}})} 2^{-\eta n-1}n^{m}
\bigr\} \bigr)
\\
&& \qquad \leq 2^{(2Q-1)n}5n^q\exp \bigl\{-c2^{(\bar{\eta}_c-\eta)(\beta+1)n/(2\beta)}
\bigr\}
\\
&& \qquad \leq \exp \bigl\{-2^{(\bar{\eta}_c-\eta)(\beta+1)n/(4\beta)} \bigr\},
\end{eqnarray*}
for all $n$ sufficiently large.
Using Borel--Cantelli, we get that with probability one
\[
\bigcup_{k=2n^q+2}^{2^n-1}\bigl
\{A^\varepsilon\cap A^{(n)}_k\cap\bigl\{
\bL_{n,l^n_k,r^n_k}^+\leq C_{(\fontsize{8.36}{10.36}{\selectfont\ref{eq420}})} 2^{-\eta n-1}n^{m}
\bigr\}\bigr\}
\]
occurs only a finite number of times. Let $x\in \tilde{J}_{\eta,1}$ be arbitrary.
By definition of $\tilde{J}_{\eta,1}$, there
exists a (random) sequence $\{(n_j, k_j)\}$ such that
\[
x\in J^{(n_j)}_{k_j,1}\quad \mbox{and}\quad
\1_{A_{k_j}^{(n_j)}}=1 \qquad \forall j\geq 1.
\]
Therefore, on the set $A^{\varepsilon}$ we have
\[
\bL_{n,l_{k_j}^{n_j},r_{k_j}^{n_j}}^+ > C_{(\fontsize{8.36}{10.36}{\selectfont\ref{eq420}})} 2^{-\eta n_j-1}n_j^{m},
\]
for all $j$ sufficiently large and all $r_{n_j}^{k_j}\in [\tilde{k}_{n_j}(x)-R,\tilde{k}_{n_j}(x) ]$.
\end{pf}

\subsection{Effect of compensation}
\label{sec4.3}

If we recall~(\ref{eq18102}), then Lemma~\ref{lower4} implies that it is maybe possible
to destroy the H\"older continuity, of any index greater than $\eta$, of the process
on the set $\tilde J_{\eta,1}$. For this purpose, we use processes $L^+_{n,k,l}$.
It is also clear from~(\ref{eq18102}) that in addition
one should show
that (loosely speaking) on a ``significant'' part of  $\tilde J_{\eta,1}$ there is no compensation of
``big'' values of $L^+$  by ``big'' values of $L^-$.

First, fix arbitrary positive constants $\rho, c, \nu$ such that
%
\begin{equation}
\label{eq1071} \hspace*{4pt}\rho < 10^{-2} \gamma, \qquad \nu\in \biggl(
\frac{\alpha\gamma+5\rho}{\eta_{\mathrm{c}}},10^{-1} \biggr),\qquad c\in \biggl(
\frac{10}{2-\eta}, \frac{1}{10\rho} \biggr).
\end{equation}

Define
%
\begin{eqnarray}
\nonumber
G_{k}^{(n)} &:=& \biggl\{\mbox{there exist at
least two jumps of }M,\mbox{ of the form $r\delta_{(s,y)}$,}
\\
&& \hspace*{6pt}\label{eq18104} \mbox{satisfying } r\geq 2^{-(\eta+1+2\rho+2c\rho)n}, s
\in\bigl[t-2^{-\alpha(1-c\rho)n},t-2^{-\alpha(1+c\rho)n}\bigr)
\\
\nonumber
&& \hspace*{60pt}\mbox{and }y\in \biggl[\frac{k}{2^n}-2^{-n(1-c\rho)(1-\nu)},
\frac{k+1}{2^n}+2^{-n(1-c\rho)(1-\nu)} \biggr] \biggr\}
\end{eqnarray}
and
\[
\tilde{G}_\eta:=\limsup_{n\to\infty}\bigcup
_{k=0}^{2^n-1} I^{(n)}_k{
\mathsf 1}_{G_k^{(n)}}.
\]
Informally, $\tilde{G}_{\eta}$ is such that in certain proximity of every
$x\in\tilde{G}_{\eta}$ there are at least two ``big'' jumps of $M$. If
one of the jumps, appears in $L^+$, and another in $L^-$, they may compensate
each other; however, in the next lemma we will show that the Hausdorff dimension
of $\tilde G_{\eta}$ is small.

\begin{lemma}
\label{compens}
On $A^{\varepsilon}$,
\[
\dim(\tilde{G}_\eta)\leq \bigl(2(\beta+1) (\eta-
\eta_c)-1 \bigr)^++2\bigl(\nu+8(c+1)\rho\bigr),\qquad \pr\mbox{-a.s.}
\]
\end{lemma}

\begin{pf}
On the event $O_n^{c}$, we have the following upper bound for the intensity of the jumps in $G^{(n)}_k$:
\begin{eqnarray*}
&& \int_{t-2^{-\alpha(1-c\rho)n}}^{t-2^{-\alpha(1+c\rho)n}}\,ds X_s \biggl(
\biggl[\frac{k}{2^n}-2^{-n(1-c\rho)(1-\nu)},\frac{k+1}{2^n}+2^{-n(1-c\rho)(1-\nu)}
\biggr] \biggr)
\\
&& \quad {}\times\int_{2^{-(\eta+1+2\rho+2c\rho)n}} \varrho r^{-2-\beta}
\,dr
\\
&& \qquad \leq C 2^{n-n(1-c\rho)(1-\nu)}2^{-n}n^{2m}2^{-\alpha(1-\rho c)n}2^{(\beta+1)(\eta+1+2\rho+2c\rho)}
\\
&& \qquad \leq C 2^{-n+n(\beta+1)(\eta-\eta_c)+\delta n},
\end{eqnarray*}
where $\delta=\nu+7(c+1)\rho$.
Since the number of such jumps can be represented by means of a time-changed standard Poisson process,
the probability to have  at least two such jumps is bounded by the square of the above bound, that is,
\[
\mathbf{P}\bigl(O_n^{c}\cap G^{(n)}_k
\bigr)\leq C2^{-2n+2n(\beta+1)(\eta-\eta_c)+2\delta n}=:p^{(n)}.
\]
Combining this bound with Lemma~\ref{lower0} and the Markov inequality, we get
\begin{eqnarray*}
\mathbf{P} \Biggl(\sum_{k=0}^{2^{n}-1}{\mathsf
1}_{G_k^{(n)}}\geq 2^{n+\varepsilon n}p^{(n)} \Biggr) &\leq &
\mathbf{P}(O_n)+ \mathbf{P} \Biggl(\sum_{k=0}^{2^{n}-1}{
\mathsf 1}_{G_k^{(n)}}\geq 2^{n+\varepsilon n}p^{(n)};O_n^{c}
\Biggr)
\\
&\leq & Cn^{-m\alpha/3}+\frac{2^n\mathbf{P}(G^{(n)}_1)}{2^{n+\varepsilon n}p^{(n)}}
\\
&\leq & Cn^{-m\alpha/3}+2^{-\varepsilon n}.
\end{eqnarray*}

If $2(\beta+1)(\eta-\eta_c)+2\delta<1$, then, choosing $\varepsilon$ sufficiently small, we obtain
\[
\mathbf{P} \Biggl(\sum_{k=0}^{2^{n}-1}{\mathsf
1}_{G_k^{(n)}}\geq 1 \Biggr) \leq C2^{-m n}+2^{-\varepsilon n}.
\]
Applying finally Borel--Cantelli, we conclude that $\tilde{G}_\eta=\varnothing$ almost surely.
In particular, $\dim (\tilde{G}_\eta)=0$ with probability one.

Assume now that $2(\beta+1)(\eta-\eta_c)+2\delta\geq1$. Applying Borel--Cantelli once again, we see that
the number of indices $k$ with ${\mathsf 1}_{G_k^{(n)}}=1$ is bounded by $2^{n+\varepsilon n}p^{(n)}$.
Noting that $\tilde{G}_\eta$ can be covered by $\bigcup_{k=1}^{2^n} I^{(n)}_k{\mathsf 1}_{G_k^{(n)}}$ and
\[
\sum_{n=1}^\infty 2^{n+\varepsilon n}p^{(n)}2^{-\theta n}<
\infty \qquad \mbox{for all }\theta>2(\beta+1) (\eta-\eta_c)+2
\delta-1+\varepsilon,
\]
we infer that
\[
\operatorname{dim}(\tilde{G}_\eta)\leq2(\beta+1) (\eta-
\eta_c)+2\delta-1+\varepsilon.
\]
Letting $\varepsilon\to0$, we get the desired result.
\end{pf}

The remaining part of the subsection is devoted to the proof of the fact that,
on the set $(\tilde{J}_{\eta,1}\setminus S_{\eta-2\rho})\setminus\tilde{G}_{\eta}$
``enough'' of  big values of
$\bL^+_{n,l,k}$ cannot be compensated by $\bL^-_{n,l,k}$. This will lead later to the
desired
upper bound for H\"older exponents for points from
$(\tilde{J}_{\eta,1}\setminus S_{\eta-2\rho})\setminus\tilde{G}_{\eta}$.

As usual, we start with the analysis of jumps. For $N\geq 1$, define
\begin{eqnarray}
F^{(N)}_{k} &:=& \bigl\{\bDelta X_s(y)\geq
(t-s)^{{1}/({1+\beta})-\gamma}\bigl|y-(k+1)2^{-N}\bigr|^{\eta-\eta_c}
\nonumber\\
&&\hspace*{78pt}{}\mbox{for some } s\geq t-2^{-\alpha N}, y\in I_k^{(N)}
\bigr\},\nonumber\\
\eqntext{k=0,\ldots, 2^{N}-1,}
\end{eqnarray}
and
\[
F^{(N)}:=\bigcup_{k=0}^{2^{N}-1}
\biggl(\bigcap_{j<R}F_{k+j}^{(N)}
\biggr).
\]
Note that we use $N$ (and not $n$) above, and at some point we will take $N=Qn$.

\begin{lemma}
\label{lastlemma}
For every $R>(1-(1+\beta)(\eta-\eta_c))^{-1}$,
there exists a constant $C=C(R)$ such that
\[
\mathbf{P}\bigl(F^{(N)}\bigr)\leq C N^{-m\alpha/3+1}.
\]
\end{lemma}

\begin{pf}
It follows from Lemma~\ref{lower0} that
\[
\mathbf{P} \biggl(\bigcup_{n\geq N}O_n \biggr)\leq\sum
_{n=N}^\infty\mathbf{P}(O_n)\leq
CN^{-m\alpha/3+1}.
\]
Therefore,
%
\begin{eqnarray}
\mathbf{P} \bigl(F^{(N)} \bigr)&\leq& \mathbf{P} \biggl(F^{(N)}
\cap \biggl(\bigcap_{n\geq N}O_n^c \biggr) \biggr) +
\mathbf{P} \biggl(\bigcup_{n\geq N}O_n \biggr)
\nonumber
\\[-8pt]
\label{ll1}
\\[-8pt]
\nonumber
&\leq & \sum_{k=0}^{2^N-R-1}
\mathbf{P} \biggl( \biggl(\bigcap_{j<R}F_{k+j}^{(N)}
\biggr)\cap \biggl( \bigcap_{n\geq N}O_n^c \biggr) \biggr)
+CN^{-m\alpha/3+1}.
\end{eqnarray}
Consider a jump characterized by the triple $(y,s,r)$. We first assume that
\[
-(t-s)^{1/\alpha}<y-(k+1)2^{-n}<0.
\]
This jump affects $F_k^{(N)}$ if and only if $r>(t-s)^{{1}/({1+\beta})-\gamma+
({\eta-\eta_c})/{\alpha}}$.
If we consider $s\in[t-2^{-\alpha j}, t-2^{-\alpha(j+1)})$, then $r$ should be greater than\break
$2^{-\alpha(j+1)({1}/({1+\beta})-\gamma+({\eta-\eta_c})/{\alpha})}$.
Since
\[
\sup_{s\geq t-2^{-\alpha j}}X_u \bigl(\bigl[(k+1)2^{-n}-2^{-j},(k+1)2^{-n}
\bigr) \bigr)\leq j^{2m}2^{-j}
\]
on the event $O_j^c$, we have the following bound for the intensity of jumps described above:
\begin{eqnarray}
\nonumber
&&\int_{t-2^{-\alpha j}}^{t-2^{-\alpha (j+1)}} duX_u
\bigl(\bigl[(k+1)2^{-n}-2^{-j},(k+1)2^{-n} \bigr) \bigr)
\\
\nonumber
&&\quad {}\times\int_{2^{-\alpha(j+1)({1}/({1+\beta})-\gamma+({\eta-\eta_c})/{\alpha})}}^\infty
\varrho r^{-2-\beta}\,dr
\nonumber
\\[-8pt]
\label{ll2}
\\[-8pt]
\nonumber
&& \qquad \leq Cj^{2m}2^{-j(\alpha+1)}2^{j(\alpha-\gamma\alpha(\beta+1)+(\eta-\eta_c)(\beta+1)}
\\
&&\qquad =Cj^{2m}2^{-j(1-(\eta-\eta_c)(\beta+1))-j\gamma\alpha(\beta+1)}.
\nonumber
\end{eqnarray}
If $(k+1)2^{-N}-y\in[a2^{-j},(a+1)2^{-j})$ with some $a\geq1$, then $r$ should\vspace*{2pt} be bigger than
$2^{-\alpha(j+1)({1}/({1+\beta})-\gamma)}a^{\eta-\eta_c}2^{-j(\eta-\eta_c)}$. Then, on the event $O_j^c$,
\begin{eqnarray}
\nonumber
&& \int_{t-2^{-\alpha j}}^{t-2^{-\alpha (j+1)}}
\,duX_u \bigl(\bigl[a2^{-j},(a+1)2^{-j} \bigr) \bigr)
\\
\nonumber
&&\quad {}\times\int_{2^{-\alpha(j+1)({1}/({1+\beta})-\gamma)}a^{\eta-\eta_c}2^{-j(\eta-\eta_c)}}^\infty
\varrho r^{-2-\beta}\,dr
\nonumber
\\[-8pt]
\label{ll3}
\\[-8pt]
\nonumber
&& \qquad \leq Cj^{2m}2^{-j(\alpha+1)}2^{j(\alpha-\gamma\alpha(\beta+1)+(\eta-\eta_c)(\beta+1)}a^{-(\eta-\eta_c)(\beta+1)}
\\
\nonumber
&& \qquad =Cj^{2m}2^{-j(1-(\eta-\eta_c)(\beta+1))-j\gamma\alpha(\beta+1)}a^{-(\eta-\eta_c)(\beta+1)}.
\end{eqnarray}
Combining (\ref{ll2}), (\ref{ll3}) and noting that we can cover the interval $I_k^{(N)}$ by
the union of intervals
$[(k+1)2^{-N}-(a+1)2^{-j}, (k+1)2^{-N}-a2^{-j})$ with $a<2^{j-N}$, we see that the intensity of jumps with
$y\in I_k^{(N)}$, $s\geq t-2^{-\alpha N}$ is bounded by\looseness=1
\begin{eqnarray*}
&& \sum_{j=N}^\infty Cj^{2m}2^{-j(1-(\eta-\eta_c)(\beta+1))-j\gamma\alpha(\beta+1)}
\Biggl(1+\sum_{a=1}^{2^{j-N}-1}a^{-(\eta-\eta_c)(\beta+1)}
\Biggr)
\\
&& \qquad \leq C\sum_{j=N}^\infty
j^{2m}2^{-j(1-(\eta-\eta_c)(\beta+1))-j\gamma\alpha(\beta+1)} \bigl(2^{j-N} \bigr)^{1-(\eta-\eta_c)(\beta+1)}
\\
&& \qquad =C \bigl(2^{-N} \bigr)^{1-(\eta-\eta_c)(\beta+1)}\sum
_{j=N}^\infty j^{2m}2^{-j\gamma\alpha(\beta+1)}
\\
&& \qquad \leq C \bigl(2^{-N} \bigr)^{1-(\eta-\eta_c)(\beta+1)}.
\end{eqnarray*}
This implies that
\[
\mathbf{P} \biggl(F_k^{(N)}\cap \biggl(
\bigcap_{n\geq N}O_n^c \biggr) \biggr)\leq C
\bigl(2^{-N} \bigr)^{1-(\eta-\eta_c)(\beta+1)}.
\]
Since the jumps can be represented by a time-changed Poisson process, we then get
\[
\mathbf{P} \biggl( \biggl(\bigcap_{j<R}F_{k+j}^{(N)}
\biggr)\cap \biggl(\bigcap_{n\geq N}O_n^c \biggr) \biggr)
\leq C\bigl(2^{-N} \bigr)^{R(1-(\eta-\eta_c)(\beta+1))}.
\]
Applying this bound to summands in (\ref{ll1}), we complete the proof of the lemma.
\end{pf}

From this lemma and the Borel--Cantelli lemma, we obtain:

\begin{corollary}\label{corrR}
Let $R$ be as in the previous lemma.
For $\pr$-a.s. $\omega\in A^{\varepsilon}$ there exists
$N_{\fontsize{8.36}{10.36}{\selectfont\ref{corrR}}}=N_{\fontsize{8.36}{10.36}{\selectfont\ref{corrR}}}(\omega)$ such
that for every $N\geq N_{\fontsize{8.36}{10.36}{\selectfont\ref{corrR}}}$ and every
$k \dvtx  R\leq  k<2^N$ there exists $j=j(k,N)\in\{1,\ldots, R\}$
with $1_{F_{k-j}^{(N)}}=0$.
\end{corollary}

We need to introduce additional  notation. Let
%
\begin{equation}
k_n(x):= \bigl\lfloor 2^n x \bigr\rfloor.
\end{equation}
Recall $A^{(n)}_{k_n(x)}$, and   let  $\tilde s=s^n_{k_n(x)}$ [see (\ref{eq02121a}), (\ref{eq02121})]
be the time and $\tilde y=y^n_{k_n(x)}$ [defined below (\ref{eq02121})]
be the spatial  position of a jump described in the definition of the event $A^{(n)}_{k_n(x)}$. Then on
$A^{(n)}_{k_n(x)}$, fix
%
\begin{equation}
\label{eq2071} \tilde{l}_n(x):= \bigl\lfloor 2^{Qn}
\tilde y \bigr\rfloor.
\end{equation}
Moreover, since $Q>1$, for every $n\geq N_{\fontsize{8.36}{10.36}{\selectfont\ref{corrR}}}$ we can  define
%
\begin{equation}
\label{eq2072} \tilde{r}_n(x) = \tilde{k}_n(x)-j \bigl(
\tilde{k}_n(x), Qn \bigr),
\end{equation}
where $j(\cdot, \cdot)$ is defined in Corollary~\ref{corrR}, and recall that $ \tilde k_n(x)= k_{Qn}(x)$.

\begin{remark}
 Note that above definition of $\tilde l_n(x)$ and especially the construction of $\tilde r_n(x)$
 are crucial for the proof\hspace*{1pt} of the lower bound. In the sequel, we will show
 that for $x\in (\tilde{J}_{\eta,1}\setminus S_{\eta-2\rho})\setminus\tilde{G}_{\eta}$, there exists
  subsequence $\{n_j\}$ such that big values of
$\bL^+_{n_j,\tilde l_{n_j}(x),\tilde r_{n_j}(x)}$ are not compensated by
$\bL^-_{n_j,\tilde l_{n_j}(x),\tilde r_{n_j}(x)}$ (see Lemma~\ref{compens2} below and Lemma~\ref{lower4} above).
\end{remark}

We also will define three sets in $[0,t)\times \R$. For any $x\in \R$, set
\begin{eqnarray*}
S^1_{n,x}&:=& \bigl\{(s,y)\in [0,t)\times \R \dvtx |y-x|
\geq (t-s)^{(1-\nu)/\alpha} \bigr\},
\\
S^2_{n,x}&:=& \bigl\{(s,y)\in \bigl( [0,t )\setminus
\bigl(t-2^{-\alpha(1-c\rho)n},t-2^{-\alpha(1+c\rho)n} \bigr)\bigr)\times \R\dvtx
\\
&& \hspace*{120pt}\qquad |y-x|\leq (t-s)^{(1-\nu)/\alpha} \bigr\},
\\
S^3_{n,x}&:=& \bigl\{(s,y)\in \bigl[t-2^{-\alpha(1-c\rho)n},t-2^{-\alpha(1+c\rho)n}
\bigr]\times\R \dvtx
\\
&& \hspace*{5pt} |y-x|\leq (t-s)^{(1-\nu)/\alpha}, \bDelta X_s(y)
\leq(t-s)^{(\eta+1+3\rho)/\alpha} \bigr\},
\end{eqnarray*}
and note that the last set is random. In the next lemma, we will show that,
under certain conditions, the jumps of $L^-_{n,\tilde l_n(x),\tilde r_n(x)}$
are\vspace*{1pt} small on the above sets.
We will also need an additional piece of notation. Let
\[
S^4:=\bigl\{(s,y)\in [0,t)\times\R\dvtx \bDelta X_s
\bigl(\{y\}\bigr)>0\bigr\}
\]
be the set of points in space$\times$time where the jumps of $X$, or equivalently of $M$, occur.

\begin{lemma}\label{lem0271}
Let $\eta\in (\eta_{\mathrm{c}},\bar\eta_{\mathrm{c}})\setminus{\{1\}}$, and
$\rho, \nu, c$ be as in~(\ref{eq1071}).
For $\pr$-a.s. $\omega\in A^{\varepsilon}$, there exists
$N_{\fontsize{8.36}{10.36}{\selectfont\ref{lem0271}}}=N_{\fontsize{8.36}{10.36}{\selectfont\ref{lem0271}}}(\omega)$ such that for every
$n\geq N_{\fontsize{8.36}{10.36}{\selectfont\ref{lem0271}}}$ the following holds. Fix arbitrary\vspace*{-3pt}
$x\in (0,1)\setminus S_{\eta-2\rho}$ such that ${\mathsf 1}_{A^{(n)}_{k_n(x)}}=1$.
Then there exists a constant $ C_{(\fontsize{8.36}{10.36}{\selectfont\ref{compens2.3}})}= C_{(\fontsize{8.36}{10.36}{\selectfont\ref{compens2.3}})}(\rho,\nu,c)$
such that for any $(s,y)\in  (\bigcup_{i=1}^3 S^i_{n,x} )\cap S^4$, we have
%
\begin{equation}
\label{compens2.3} \bDelta L^-_{n,\tilde l_n(x),\tilde r_n(x)} \bigl(T_{-}^{n,\tilde l_n(x),\tilde r_n(x)}(s)
\bigr)\leq C_{(\fontsize{8.36}{10.36}{\selectfont\ref{compens2.3}})}2^{-(\eta+2\rho)n}.
\end{equation}
\end{lemma}

\begin{pf}
Fix some $\omega\in A^{\varepsilon}$ and choose $ N_{\fontsize{8.36}{10.36}{\selectfont\ref{lem0271}}}\geq
N_{\fontsize{8.36}{10.36}{\selectfont\ref{corrR}}}$; the choice of $N_{\fontsize{8.36}{10.36}{\selectfont\ref{lem0271}}}$ will be clear from the proof.
Take arbitrary $n\geq N_{\fontsize{8.36}{10.36}{\selectfont\ref{lem0271}}}$.
Fix also some $x\in (0,1)\setminus S_{\eta-2\rho}$ satisfying ${\mathsf 1}_{A^{(n)}_{k_n(x)}}=1$.
Recall~(\ref{eq2071}), (\ref{eq2072}) and in what follows to simplify the notation  denote
\[
l=\tilde l_n(x),\qquad r=\tilde r_n(x).
\]

It is clear that a jump, which appears in the definition of $A^{(n)}_{k_n(x)}$, does not produce a jump of $L^-_{n,l,r}$.

Recall that $x\in (0,1)\setminus S_{\eta-2\rho}$ means that, for any $y\in \R$,
%
\begin{equation}
\label{s-rho} \bDelta X_s(y)\leq (t-s)^{{1}/({\beta+1})-\gamma}|y-x|^{\eta-2\rho-\eta_c}.
\end{equation}

We will treat the three regions $S^i_{n,x}, i=1,2,3$ separately.
\begin{longlist}[(ii)]
\item[(i)]  Let  $(s,y)\in S^1_{n,x}\cap S^4$, that is,
%
\begin{equation}
\label{y-condition} |y-x|\geq (t-s)^{(1-\nu)/\alpha}.
\end{equation}
First assume $|y-x|\leq n^q2^{-n}$.
We will consider the cases of $\eta<1$ and $\eta>1$ separately. We start with the following.
\begin{longlist}[\textit{Case} $\eta<1$.]
\item[\textit{Case} $\eta<1$.]
For all $x_1,x_2\in\mathsf{R}$,
\[
\bigl(p^{\alpha,\eta}_{t-s}(x_1,x_2)
\bigr)^{-}= \bigl(p^{\alpha}_{t-s}(x_1)-p^{\alpha}_{t-s}(x_1)
\bigr)^{-} \leq \frac{p_1^\alpha(0)}{(t-s)^{1/\alpha}}.
\]
Combining this inequality with (\ref{s-rho}) and (\ref{y-condition}), we get
\begin{eqnarray*}
\bDelta L^-_{n,l,r} \bigl(T_{-}^{n,l,r}(s) \bigr)&
\leq & p_1^\alpha(0) (t-s)^{{1}/({\beta+1})-\gamma}|y-x|^{\eta-2\rho-\eta_c}(t-s)^{-1/\alpha}
\\
&=& p_1^\alpha(0) (t-s)^{(\eta_c-\alpha\gamma)/\alpha}|y-x|^{\eta-2\rho-\eta_c}
\\
&\leq & p_1^\alpha(0)|y-x|^{({\eta_c-\alpha\gamma})/({1-\nu})+\eta-\eta_c-2\rho}.
\end{eqnarray*}
We choose $N_{\fontsize{8.36}{10.36}{\selectfont\ref{lem0271}}}$ sufficiently large such that, for any $\nu\in  (\frac{\alpha\gamma+5\rho}{\eta_{\mathrm{c}}},10^{-1} )$ and $|y-x|\leq n^q2^{-n}$,
(\ref{compens2.3}) holds, for all $n\geq N_{\fontsize{8.36}{10.36}{\selectfont\ref{lem0271}}}$.

\item[\textit{Case} $\eta>1$.] Here, we have
\begin{eqnarray*}
&& p^{\alpha,\eta}_{t-s} \bigl(l2^{-Qn}-y,r2^{-Qn}-y
\bigr)
\\
&& \qquad = p^\alpha_{t-s} \bigl(l2^{-Qn}-y
\bigr)-p^\alpha_{t-s} \bigl(r2^{-Qn}-y \bigr)
+(r-l)2^{-Qn}p^{\alpha,\prime}_{t-s} \bigl(r2^{-Qn}-y
\bigr).
\end{eqnarray*}
Note that $p^{\alpha,\prime}_{t-s}(z)\geq0$ for all $z\leq0$. Consequently,
%
\begin{eqnarray}
&& \bigl(p^{\alpha,\eta}_{t-s} \bigl(l2^{-Qn}-y,r2^{-Qn}-y
\bigr) \bigr)^{-} \nonumber\\
&& \label{compens2.3a}\qquad \leq  \bigl(p^{\alpha}_{t-s}
\bigl(l2^{-Qn}-y \bigr)- p^{\alpha}_{t-s}
\bigl(r2^{-Qn}-y \bigr) \bigr)^{-}
\\
\nonumber
&& \qquad \leq  \frac{p_1^\alpha(0)}{(t-s)^{1/\alpha}}
\end{eqnarray}
for all $y\geq r2^{-Qn}$.

In the complementary case $y< r2^{-Qn}$, one can easily get
%
\begin{eqnarray}
&& \bigl(p^{\alpha,\eta}_{t-s} \bigl(l2^{-Qn}-y,r2^{-Qn}-y
\bigr) \bigr)^{-}
\nonumber
\\[-8pt]
\label{compens2.3b}
\\[-8pt]
\nonumber
&& \qquad \leq\frac{p_1^\alpha(0)}{(t-s)^{1/\alpha}} + \bigl\llvert (r-l)2^{-Qn}p^{\alpha,\prime}_{t-s}
\bigl(r2^{-Qn}-y \bigr) \bigr\rrvert.
\end{eqnarray}
If $y\leq (r-1)2^{-Qn}$, then $r2^{-Qn}-y\geq (x-y)/(R+1)$. Thus, using the bound
$p^{\alpha,\prime}_1(z)\leq C|z|^{-\alpha-2}$ and the scaling property, we obtain
%
\begin{eqnarray}
-p^{\alpha,\prime}_{t-s} \bigl(r2^{-Qn}-y \bigr) &\leq &
C(t-s)^{-2/\alpha} \biggl(\frac{r2^{-Qn}-y}{(t-s)^{1/\alpha}} \biggr)^{-\alpha-2}
\nonumber
\\[-8pt]
\label{compens2.3c}
\\[-8pt]
\nonumber
&\leq & CR^{\alpha+2}(t-s)|x-y|^{-\alpha-2}.
\end{eqnarray}
From this, (\ref{compens2.3a}) and (\ref{compens2.3b}) we conclude that, for all $y$ satisfying
$|y-r2^{-Qn}|>2^{-Qn}$,
\begin{eqnarray*}
&& \bigl(p^{\alpha,\eta}_{t-s} \bigl(l2^{-Qn}-y,r2^{-Qn}-y
\bigr) \bigr)^{-}
\\
&&\qquad \leq C \bigl((t-s)^{-1/\alpha}+(r-l)2^{-Qn}(t-s)|x-y|^{-\alpha-2}
\bigr).
\end{eqnarray*}
Combining this with (\ref{s-rho}) we conclude that the corresponding
jump\break
$\bDelta L_{n,l,r}^-(T_{-}^{n,l,r}(s))$
is bounded by
\[
C(t-s)^{{1}/({\beta+1})-\gamma}|y-x|^{\eta-2\rho-\eta_c} \bigl((t-s)^{-1/\alpha}+(r-l)2^{-Qn}(t-s)|x-y|^{-\alpha-2}
\bigr).
\]
Taking into account (\ref{y-condition}), we see that the expression on the right-hand side does not exceed
\[
C \bigl(|y-x|^{({\eta_c-\alpha\gamma})/({1-\nu})+\eta-\eta_c-2\rho}+ (r-l)2^{-Qn}|y-x|^{({\eta-1})/({1-\nu})+\nu\alpha-2\rho}
\bigr).
\]
Now it is easy to see that (\ref{compens2.3}) remain valid for $\eta>1$ under additional assumption
$y\leq (r-1)2^{-Qn}$, or  $y\geq r2^{-Qn}$.

Now we will take care of the case  $y\in((r-1)2^{-Qn},r2^{-Qn})$.
By Corollary~\ref{corrR} and our definition of $r=\tilde{r}_n(x) $ [recall again (\ref{eq2072}) and
$n\geq N_{\fontsize{8.36}{10.36}{\selectfont\ref{lem0271}}}\geq N_{\fontsize{8.36}{10.36}{\selectfont\ref{corrR}}}$, $Q>1$],
we obtain
\begin{equation}
\label{compens2.3d} \bDelta X_s(y)\leq (t-s)^{{1}/({\beta+1})-\gamma}\bigl|y-r2^{-Qn}\bigr|^{\eta-\eta_c}
\end{equation}
for all $y\in((r-1)2^{-Qn},r2^{-Qn})$ and $s\geq t-2^{-\alpha Qn}$. It is clear that $y\in((r-1)2^{-Qn},r2^{-Qn})$
implies that $|y-x|\leq (R+1)2^{-Qn}$. From this and (\ref{y-condition}), we infer that
$(t-s)\leq (R+1)^{\alpha/(1-\nu)}2^{-\alpha Qn/(1-\nu)}$ and, consequently, $s\geq t-2^{-\alpha Qn}$ for
all sufficiently large $n$. Repeating all the arguments after (\ref{compens2.3b}) and using (\ref{compens2.3d})
instead of (\ref{s-rho}), we obtain (\ref{compens2.3}).

Summarizing, (\ref{compens2.3}) is valid
for all $|y-x|\leq n^q2^{-n}$.

 In case $|y-x|\geq n^q2^{-n}$, we apply Corollary~\ref{cor14101} if
$\eta>1$, or Lemma~\ref{L1old} if $\eta<1$ with
$\delta=\eta+3\rho$ (recall the bounds on $\rho$ and $\gamma$ to get that $\delta<1$ if $\eta<1$ and $\delta<2$ if $\eta> 1$)
to get
\[
\bigl\llvert \tilde{p}^{\alpha,\eta}_{t-s} \bigl(l2^{-Qn}-y,r2^{-Qn}-y
\bigr)\bigr\rrvert \leq C\frac{(r-l)^\delta2^{-Q\delta n}}{(t-s)^{(\delta+1)/\alpha}}p_1^\alpha
\biggl(\frac{y-r2^{-Qn}}{(t-s)^{1/\alpha}} \biggr).
\]
Since $(r-l)2^{-Qn}\leq 4n^q2^{-n}$ and $r2^{-Qn}\leq x$, we then have
\[
\bigl\llvert \tilde{p}^{\alpha,\eta}_{t-s} \bigl(l2^{-Qn}-y,r2^{-Qn}-y
\bigr)\bigr\rrvert \leq C\frac{n^{q\delta}2^{-\delta n}}{(t-s)^{(1+\delta)/\alpha}}p_1^\alpha
\biggl(\frac{y-x}{(t-s)^{1/\alpha}} \biggr).
\]
From this bound and~(\ref{eq326}),
 we obtain
\begin{eqnarray*}
&&\bDelta L^-_{n,l,r} \bigl(T_{-}^{n,l,r}(s) \bigr)
\\
&& \qquad \leq C (t-s)^{{1}/({\beta+1})-\gamma}|y-x|^{\eta-\eta_c-2\rho}\frac{n^{q\delta}2^{-\delta n}}{(t-s)^{(\delta+1)/\alpha}}
p_1^\alpha \biggl(\frac{y-x}{(t-s)^{1/\alpha}} \biggr)
\\
&& \qquad \leq Cn^{q\delta}2^{-\delta n}(t-s)^{{1}/({\beta+1})-\gamma-({\delta+1})/{\alpha}+({\alpha+1})/{\alpha}}
|y-x|^{\eta-\eta_c-2\rho-\alpha-1}.
\end{eqnarray*}
As a result, for $|y-x|\geq(t-s)^{(1-\nu)/\alpha}$ we have
\begin{eqnarray*}
&& \bDelta L^-_{n,l,r} \bigl(T_{-}^{n,l,r}(s) \bigr)
\\
&& \qquad \leq C n^{q\delta}2^{-\delta n} \bigl((t-s)^{1/\alpha}
\bigr)^{{\alpha}/({\beta+1})-\alpha\gamma-\delta+\alpha+(\eta-\eta_c-2\rho-\alpha-1)(1-\nu)}
\\
&& \qquad \leq C 2^{-(\eta+2\rho)n} \bigl((t-s)^{1/\alpha}
\bigr)^{\eta_c+1-\alpha\gamma-\eta-3\rho+\alpha+(\eta-\eta_c-2\rho-\alpha-1)(1-\nu)}
\\
&& \qquad = C 2^{-(\eta+2\rho)n} \bigl((t-s)^{1/\alpha} \bigr)^{\nu(\alpha+1-\eta+\eta_c+2\rho)-\alpha\gamma-5\rho}.
\end{eqnarray*}
Finally, recall that $\nu\geq \frac{5\rho+\alpha\gamma}{\eta_{\mathrm{c}}}$, and then  (\ref{compens2.3})
holds with an appropriate constant~$C_{(\fontsize{8.36}{10.36}{\selectfont\ref{compens2.3}})}$.
\end{longlist}

\item[(ii)] Let  $(s,y)\in S^2_{n,x}\cap S^4$.
We start with the subset of $S^2_{n,x}$ where  $|y-x|\leq(t-s)^{(1-\nu)/\alpha}$   and  $s\leq t-2^{-\alpha(1-c\rho)n}$.

If $\eta<1$ then, applying Lemma~\ref{L1old} with $\delta=1$, we obtain
\begin{eqnarray*}
\bDelta L^-_{n,l,r} \bigl(T_{-}^{n,l,r}(s) \bigr)&
\leq& C(t-s)^{{1}/({\beta+1})-\gamma}|y-x|^{\eta-\eta_c-2\rho}\frac{4n^q2^{-n}}{(t-s)^{2/\alpha}}
\\
&\leq & C n^q2^{-n}(t-s)^{{1}/({\beta+1})-\gamma-2/\alpha}(t-s)^{(1-\nu)(\eta-\eta_c-2\rho)/\alpha}
\\
&=& C n^q2^{-n} \bigl((t-s)^{1/\alpha}
\bigr)^{{\alpha}/({\beta+1})-\alpha\gamma-2+(1-\nu)(\eta-\eta_c-2\rho)}
\\
&=& C n^q2^{-n} \bigl((t-s)^{1/\alpha}
\bigr)^{\eta-1-\alpha\gamma-2\rho-\nu(\eta-\eta_c-2\rho)}
\\
&\leq & C n^q2^{-n} \bigl(2^{-n(1-c\rho)}
\bigr)^{\eta-1-\alpha\gamma-2\rho-\nu(\eta-\eta_c-2\rho)}.
\end{eqnarray*}
If $\eta>1$, then we can apply Corollary~\ref{cor14101} with $\delta=2$, which gives
\begin{eqnarray*}
\bDelta L^-_{n,l,r} \bigl(T_{-}^{n,l,r}(s) \bigr)&
\leq & C(t-s)^{{1}/({\beta+1})-\gamma}|y-x|^{\eta-\eta_c-2\rho}\frac{16n^{2q}2^{-2n}}{(t-s)^{3/\alpha}}
\\
&\leq & C n^{2q}2^{-2n}(t-s)^{{1}/({\beta+1})-\gamma-3/\alpha}(t-s)^{(1-\nu)(\eta-\eta_c-2\rho)/\alpha}
\\
&=& C n^{2q}2^{-2n} \bigl((t-s)^{1/\alpha}
\bigr)^{{\alpha}/({\beta+1})-\alpha\gamma-3+(1-\nu)(\eta-\eta_c-2\rho)}
\\
&=& C n^{2q}2^{-2n} \bigl((t-s)^{1/\alpha}
\bigr)^{\eta-2-\alpha\gamma-2\rho-\nu(\eta-\eta_c-2\rho)}
\\
&\leq & C n^{2q}2^{-2n} \bigl(2^{-n(1-c\rho)}
\bigr)^{\eta-2-\alpha\gamma-2\rho-\nu(\eta-\eta_c-2\rho)}.
\end{eqnarray*}
Hence, for $\eta\in (\eta_{\mathrm{c}},\bar{\eta}_{\mathrm{c}})\setminus{\{1\}}$, and
with $c,\rho,\nu$ as in (\ref{eq1071}), we immediately get~(\ref{compens2.3}) with an appropriate constant $C_{(\fontsize{8.36}{10.36}{\selectfont\ref{compens2.3}})}$.

Now we consider the complimentary subset of  $S^2_{n,x}$, where
\[
|y-x|\leq(t-s)^{(1-\nu)/\alpha}\quad \mbox{and}\quad s\geq t-2^{-\alpha(1+c\rho)n}.
\]
It follows from the definition of $\tilde{p}^{\alpha,\eta}_{t-s}$ that, for $\eta>1$,
%
\begin{eqnarray}
&& \bigl\llvert \tilde{p}^{\alpha,\eta}_{t-s}
\bigl(l2^{-Qn}-y,r2^{-Qn}-y \bigr) \bigr\rrvert
\nonumber
\\[-8pt]
\label{compens2.5a}\\[-8pt]
\nonumber
&& \qquad\leq
\frac{2p_1^\alpha(0)}{(t-s)^{1/\alpha}}+(r-l)2^{-Qn}{\sup_{z}\biggl|\frac{\partial}{\partial z}p_1^\alpha(z)\biggr|}\Big/{(t-s)^{2/\alpha}}.
\end{eqnarray}
With this and recalling that $(r-l)2^{-Qn}\leq 4n^q2^{-n}$, we obtain
\begin{eqnarray*}
\bDelta L^-_{n,l,r} \bigl(T_{-}^{n,l,r}(s) \bigr)&
\leq& C(t-s)^{{1}/({\beta+1})-\gamma}|y-x|^{\eta-\eta_c-2\rho}
\\
&&{}\times \bigl((t-s)^{-1/\alpha}+n^q2^{-n}(t-s)^{-2/\alpha}
\bigr)
\\
&\leq & Cn^q2^{-n}(t-s)^{{1}/({\beta+1})-\gamma-2/\alpha}|y-x|^{\eta-\eta_c-2\rho}
\\
&\leq & Cn^q2^{-n} \bigl((t-s)^{1/\alpha}
\bigr)^{\eta_c-1-\alpha\gamma+(1-\nu)(\eta-\eta_c-2\rho)}
\\
&\leq & Cn^q2^{-n} \bigl(2^{-(1+c\rho)n}
\bigr)^{\eta-1-\alpha\gamma-2\rho-\nu(\eta-\eta_c-2\rho)}.
\end{eqnarray*}
Again, with $c,\rho,\nu$ as in (\ref{eq1071}), we immediately get
(\ref{compens2.3}) with an appropriate constant $C_{(\fontsize{8.36}{10.36}{\selectfont\ref{compens2.3}})}$.
If $\eta<1$ then, instead of (\ref{compens2.5a}), we have a simpler inequality
\[
\bigl\llvert \tilde{p}^{\alpha,\eta}_{t-s} \bigl(l2^{-Qn}-y,r2^{-Qn}-y
\bigr)\bigr\rrvert \leq \frac{2p_1^\alpha(0)}{(t-s)^{1/\alpha}}.
\]
Thus,
\[
\bDelta L^-_{n,l,r}\bigl(T_{-}^{n,l,r}(s)\bigr)\leq
C \bigl(2^{-(1+c\rho)n} \bigr)^{\eta-\alpha\gamma-2\rho-\nu(\eta-\eta_c-2\rho)}.
\]
Consequently, (\ref{compens2.3}) holds also for $\eta<1$.

\item[(iii)] Let  $(s,y)\in S^3_{n,x}\cap S^4$.

Recall that on this set,   $\bDelta X_s(y)\leq(t-s)^{(\eta+1+3\rho)/\alpha}$.
 Then
applying Corollary~\ref{cor14101} (if $\eta>1$) or Lemma~\ref{L1old} (if $\eta< 1$)
with $\delta=\eta+3\rho$, and by using that $c,\rho,\nu$ are as in (\ref{eq1071}),
  one can easily get (\ref{compens2.3}) in this case as well.\quad\qed
\end{longlist}
\noqed\end{pf}

Recall that $\bL^-_{n,l, r} = L^-_{n,l,r}(T^{n,l,r}_{-}(t))$.
In the next lemma, we will deal with regions where
$ \{\bL^-_{n,\tilde l_n(x),\tilde r_n(x)} \}_{n\geq 1}$
may take ``big'' values infinitely often.

\begin{lemma}
\label{compens2}
For $x\in (0,1)$  define events
\[
\mathbf{B}_n(x)\equiv \bigl\{\bL^-_{n,\tilde l_n(x),\tilde r_n(x)}\geq
2^{-\eta n-1} \bigr\} \cap \{ {\mathsf 1}_{A^{(n)}_{k_n(x)}}=1 \},
\qquad n\geq 1.
\]
For any $\eta\in (\eta_{\mathrm{c}},\bar\eta_{\mathrm{c}})\setminus{\{1\}}$, we have
\[
\pr \bigl( \bigl\{x\in(0,1)\setminus S_{\eta-2\rho} \dvtx
\mathbf{B}_n(x) \mbox{ i.o.} \bigr\} \subset \tilde{G}_\eta
|A^\varepsilon \bigr)=1.
\]
\end{lemma}

\begin{pf}
First, we will show that on $\omega\in A^{\varepsilon}$
%
\begin{equation}
\label{eq18105} \bigl\{x\in (0,1)\setminus S_{\eta-2\rho}\dvtx
\mathbf{B}^*_n(x)\mbox{ i.o.} \bigr\} \subset \tilde{G}_\eta,
\end{equation}
where for $x\in (0,1)$,
\[
\mathbf{B}^*_n(x)\equiv \bigl\{\bDelta L^-_{n,\tilde l_n(x),\tilde r_n(x)}\geq
C_{(\fontsize{8.36}{10.36}{\selectfont\ref{compens2.3}})}2^{-(\eta+2\rho)n} \bigr\} \cap \{{\mathsf
1}_{A^{(n)}_{k_n(x)}}=1 \},\qquad n\ge 1.
\]
On $\omega\in A^{\varepsilon}$, take $n\geq N_{\fontsize{8.36}{10.36}{\selectfont\ref{lem0271}}}$, and
fix some $x\in (0,1)\setminus S_{\eta-2\rho}$ satisfying ${\mathsf 1}_{A^{(n)}_{k_n(x)}}=1$.
First of all, by definition of $A^{(n)}_{k_n(x)}$, if ${\mathsf 1}_{A^{(n)}_{k_n(x)}}=1$
then there exists a jump of~$M$ of the form $\tilde r\delta_{(\tilde s,\tilde y)}$
with $\tilde r,\tilde s,\tilde y$ as in $G^{(n)}_{k_n(x)}$ [see~(\ref{eq18104})].
Moreover, again by the definition of $A^{(n)}_{k_n(x)}$, the spatial position of the jump,
$\tilde y$, is in $I_{k_n(x)-2n^q-2}^{(n)}$. Hence, it is easy to see that this jump does not
contribute to the jumps of  $L^-_{n,\tilde l_n(x),\tilde r_n(x)}$
 that is,
$\bDelta L^-_{n,\tilde l_n(x),\tilde r_n(x)}(T^{n,\tilde l_n(x),\tilde r_n(x)}_{-}(\tilde s))=0$.
Thus,  we have to show that,
if $\bDelta L^-_{n,\tilde l_n(x),\tilde r_n(x)}\geq C_{(\fontsize{8.36}{10.36}{\selectfont\ref{compens2.3}})}2^{-(\eta+2\rho)n},\mbox{ and }{\mathsf 1}_{A^{(n)}_{k_n(x)}}=1$,
then there exists at least one another big jump of $M$ with properties described in
$G^{(n)}_k$.

By Lemma~\ref{lem0271}, we get that if there exists $s$ such that
\[
\bDelta L^-_{n,\tilde l_n(x),\tilde r_n(x)}\bigl(T^{n,\tilde l_n(x),\tilde r_n(x)}_{-}(s)\bigr)\geq
C_{(\fontsize{8.36}{10.36}{\selectfont\ref{compens2.3}})}2^{-(\eta+2\rho)n},
\]
then the corresponding jump of $M$, of the form  $r\delta_{(s,y)}$,
has to satisfy
\begin{eqnarray*}
|y-x| &\leq &  (t-s)^{(1-\nu)/\alpha},\qquad s\in\bigl[t-2^{-\alpha(1-c\rho)n},t-2^{-\alpha(1+c\rho)n}
\bigr], \\
 r &\geq &  (t-s)^{(\eta+1+3\rho)/\alpha}.
\end{eqnarray*}
This yields that on  $A^{\varepsilon}$ (\ref{eq18105}) holds.

Second, it follows from the second inequality in Lemma~\ref{L3} and \eqref{eq181012}, that
\[
\mathbf{P} \bigl(\bL^-_{n,l,r}\geq 2^{-\eta n-1}; \bDelta
L^-_{n,l,r}\leq C_{(\fontsize{8.36}{10.36}{\selectfont\ref{compens2.3}})}2^{-(\eta+2\rho)n} \bigr)\leq \exp\bigl
\{-c2^{2\rho n}\bigr\}
\]
for all $l,r$ satisfying $(r-l)2^{-Qn}\leq 4n^q2^{-n}$.
(Recall that $(\tilde r_n(x)-\tilde l_n(x))2^{-Qn}\leq 4n^q2^{-n}$.)

Applying now Borel--Cantelli, we conclude that, with probability one,
\[
\bigcup_{0\leq l<r\leq 2^{Qn}-1,(r-l)2^{-Qn}\leq 4n^q2^{-n}} \bigl\{\bL^-_{n,l,r}\geq
2^{-\eta n-1}; \bDelta L^-_{n,l,r}\leq C_{(\fontsize{8.36}{10.36}{\selectfont\ref{compens2.3}})}2^{-(\eta+2\rho)n}
\bigr\}
\]
occurs only finite number of times. This completes the proof of the lemma.
\end{pf}

\subsection{Proof of Proposition~\texorpdfstring{\protect\ref{prop08023}}{4.1}}

Fix arbitrary
$\eta\in(\eta_{\mathrm{c}},\bar{\eta}_{\mathrm{c}})\setminus\{1\}$. Also fix
\[
Q= \biggl[\max \biggl\{4\frac{\eta}{\eta_{\mathrm{c}}}, 4\frac{\eta}{|\eta-1|} \biggr\}+2
\biggr],
\]
where $[x]$ denotes the integer part of $x$.

It follows from Lemmas \ref{lower4} and \ref{compens2},  that
for every $x\in(\tilde{J}_{\eta,1}\setminus S_{\eta-2\rho})\setminus \tilde{G}_\eta$, for $\pr$-a.s. $\omega$ on $A^{\varepsilon}$, there exists
a (random) sequence $\{n_j\}_{j\geq 1}$ such that
\[
\bL^+_{n_j,\tilde{l}_{n_j}(x),\tilde{r}_{n_j}(x)}\geq n_j^{m}2^{-\eta n_j},
\qquad \bL^-_{n_j,\tilde{l}_{n_j}(x),\tilde{r}_{n_j}(x)}\leq 2^{-(\eta n_j-1)},
\]
for all $n_j$ sufficiently large.
This implies that, on the event $A^{\varepsilon}$, we have
%
\begin{equation}
\label{compens2.8} \liminf_{j\to\infty}2^{(\eta+\delta)n_j} \biggl\llvert
\wZ^{2,\eta}_t \biggl(\frac{\tilde{l}_{n_j}(x)}{2^{Qn_j}},\frac{\tilde{r}_{n_j}(x)}{2^{Qn_j}}
\biggr) \biggr\rrvert =\infty,
\end{equation}
for any $\delta>0$. Recall that, $X_t(\cdot)$ and $Z^2_t(\cdot)$ are H\"older continuous
with any  exponent less than $\eta_c$ at every point of $(0,1)$. Therefore, recalling that
 $Q>4\frac{\eta}{\eta_c}$, we have
%
\begin{eqnarray}
&& \lim_{j\rightarrow\infty}\sup_{x\in (0,1)}2^{(\eta+\delta)n_j}
\bigl\llvert X_t(x)-X_t \bigl(\tilde{r}_{n_j}(x)2^{-Qn_j}
\bigr) \bigr\rrvert
\nonumber
\\[-8pt]
\label{compens2.1}
\\[-8pt]
\nonumber
&& \qquad =\lim_{j\rightarrow \infty} C(\omega)
2^{-({1}/{2})Q\eta_c n_j}2^{(\eta+\delta)n_j} =0, \qquad\pr\mbox{-a.s. on }A^{\varepsilon}.
\end{eqnarray}
If $\eta<1$, then $\wZ^{2,\eta}_t(x_1,x_2)=Z^2_t(x_1)
-Z^2_t(x_2)$.
Therefore, combining (\ref{compens2.8}) and~(\ref{compens2.1}),
we conclude that
%
\begin{equation}
\label{compens2.1a}
H_{Z^2}(x)\leq \eta\qquad \mbox{for all }x\in (
\tilde{J}_{\eta,1}\setminus S_{\eta-2\rho})\setminus
\tilde{G}_\eta \ \pr\mbox{-a.s. on }A^{\varepsilon}.
\end{equation}
Assume now that $\eta>1$. In this case, we infer from (\ref{eq14104}) that $\pr$-a.s. on $A^{\varepsilon}$,
%
\begin{equation}
\label{compens2.1b}\quad \limsup_{j\to\infty} \sup_{x\in(0,1)\setminus S_{\eta-2\rho}}2^{Q(\eta-1-2\rho-2\alpha\gamma)n_j}
\bigl\llvert V' \bigl(\tilde{r}_{n_j}2^{-Qn_j}
\bigr)-V'(x) \bigr\rrvert =0.
\end{equation}
Combining (\ref{compens2.8}), (\ref{compens2.1}) and (\ref{compens2.1b}), 
and recalling that
 $Q>4\eta/(\eta-1)$, we get
\[
\liminf_{j\to\infty}2^{(\eta+\delta)n_j} \bigl\llvert
\wZ^{2,\eta}_t \bigl(2^{-Qn_j}\tilde{l}_{n_j}(x),x
\bigr)\bigr\rrvert =\infty\qquad \mbox{on }A^{\varepsilon}, \pr\mbox{-a.s.}
\]
This implies that (\ref{compens2.1a}) holds.

We know, by Lemma~\ref{P1}, that
\[
H_{Z^2}(x)\geq \eta-\alpha\gamma-2\rho \qquad \mbox{for all }x\in
(0,1)\setminus S_{\eta-2\rho}, \pr\mbox{-a.s.}
\]
This and (\ref{compens2.1a}) imply that on  $A^{\varepsilon}$, $\pr\mbox{-a.s.}$,
%
\begin{eqnarray}
 && \eta-\alpha\gamma-2\rho\leq H_{Z^2}(x)\leq \eta\qquad
\nonumber
\\[-8pt]
\label{compens2.0}\\[-8pt]
\eqntext{\mbox{for all }x\in (\tilde{J}_{\eta,1}\setminus S_{\eta-2\rho})
\setminus \tilde{G}_\eta, \forall\eta\in(\eta_{\mathrm{c}},\bar{
\eta}_{\mathrm{c}})\setminus\{1\}.}
\end{eqnarray}
It follows easily from Lemma~\ref{P2}, Corollary~\ref{lower3} and Lemma~\ref{compens} that
on  $A^{\varepsilon}$
\[
\operatorname{dim} \bigl((\tilde{J}_{\eta,1}\setminus
S_{\eta-2\rho})\setminus \tilde{G}_\eta \bigr)\geq(\beta+1) (
\eta-\eta_c), \qquad \pr\mbox{-a.s.}
\]
Thus, by (\ref{compens2.0}),
\[
\operatorname{dim}\bigl\{x \dvtx H_{Z^2}(x)\leq\eta\bigr\}\geq(
\beta+1) (\eta-\eta_c) \qquad \mbox{on } A^{\varepsilon}, \pr
\mbox{-a.s.}
\]
It is clear that
\begin{eqnarray*}
&& \bigl\{x \dvtx H_{Z^2}(x)=\eta \bigr\}\cup\bigcup
_{n=n_0}^\infty \bigl\{x \dvtx H_{Z^2}(x)\in\bigl(
\eta-n^{-1},\eta-(n+1)^{-1}\bigr] \bigr\}
\\
&& \qquad = \bigl\{x \dvtx \eta-n_0^{-1}\leq
H_{Z^2}(x)\leq\eta \bigr\}.
\end{eqnarray*}
Consequently,
\begin{eqnarray*}
&& \mathcal{H}_\eta \bigl( \bigl\{x \dvtx \eta-n_0^{-1}
\leq H_{Z^2}(x)\leq\eta \bigr\} \bigr)
\\
&& \qquad =\mathcal{H}_\eta \bigl( \bigl\{x \dvtx
H_{Z^2}(x)= \eta \bigr\} \bigr)
\\
&&\quad \qquad {}+\sum_{n=n_0}^\infty
\mathcal{H}_\eta\bigl( \bigl\{x \dvtx H_{Z^2}(x)\in \bigl(
\eta-n^{-1},\eta-(n+1)^{-1}\bigr] \bigr\} \bigr).
\end{eqnarray*}
Since the dimensions of $S_{\eta-2\rho}$ and $\tilde{G}_\eta$ are smaller than $\eta$, the $\mathcal{H}_\eta$-measure of these
sets equals zero. Applying Corollary~\ref{lower3}, we then conclude that on $ A^{\varepsilon}$
\[
\mathcal{H}_\eta\bigl((\tilde{J}_{\eta,1}\setminus
S_{\eta-2\rho})\setminus \tilde{G}_\eta\bigr)>0, \qquad \pr\mbox{-a.s.}
\]
And in view of (\ref{compens2.0}),
$\mathcal{H}_\eta(\{x \dvtx  \eta-n_0^{-1}\leq H_{Z^2}(x)\leq\eta\})>0$. Furthermore, it follows from Proposition~\ref{prop0802},
that  dimension of the set $\{x \dvtx  H_{Z^2}(x)\in(\eta-n^{-1},\eta-(n+1)^{-1}]\}$ is bounded from above by
$(\beta+1)(\eta-(n+1)^{-1}-\eta_c)$.
Hence, the definition of
 $\mathcal{H}_{\eta}$ immediately yields
\[
\mathcal{H}_\eta\bigl(\bigl\{x \dvtx H_{Z^2}(x)\in\bigl(
\eta-n^{-1},\eta-(n+1)^{-1}\bigr]\bigr\}\bigr)=0 \qquad \mbox{on }
A^{\varepsilon}, \pr\mbox{-a.s.},
\]
for all $n\geq n_0$. As a result, we have
%
\begin{equation}
\label{eq020810} \mathcal{H}_\eta \bigl( \bigl\{x \dvtx
H_{Z^2}(x)= \eta \bigr\} \bigr)>0\qquad \pr\mbox{-a.s. on }A^{\varepsilon}.
\end{equation}
Since $\varepsilon>0$ was arbitrary, this implies that~(\ref{eq020810}) is satisfied on the whole
probability space $\pr$-a.s.
From this, we get  that
\[
\operatorname{dim}\bigl\{x \dvtx H_{Z^2}(x)=\eta\bigr\}\geq(\beta+1) (
\eta-\eta_c), \qquad \pr\mbox{-a.s.}
\]

\section*{Acknowledgements}

L. Mytnik thanks the LMU Munich and V. Wachtel thanks the Technion for their hospitality.
Both authors thank an anonymous referee and an Associate Editor for their
careful reading of the manuscript and very helpful
suggestions.






\printaddresses
\end{document}